\documentclass[a4paper,11pt,reqno]{amsart}
\usepackage{amssymb,mathrsfs,enumitem}

\setlist[enumerate]{leftmargin=0pt,itemindent=32pt,nosep}
\setlist[itemize]{leftmargin=32pt,nosep}
\SetEnumitemKey{indented}{leftmargin=32pt,itemindent=0pt}
\newcommand{\first}{\setlength{\itemindent}{25pt}}

\renewcommand{\labelitemi}{\textup{\raise0.7pt\hbox{\scriptsize\textbullet\,}}}

\newtheorem{theorem}{Theorem}[section]
\newtheorem{proposition}[theorem]{Proposition}
\newtheorem{lemma}[theorem]{Lemma}
\newtheorem{corollary}[theorem]{Corollary}
\newtheorem{conjecture}[theorem]{Conjecture}

\newtheorem{conjectureAlph}{Conjecture}

\theoremstyle{definition}
\newtheorem{definition}[theorem]{Definition}
\newtheorem{algorithm}[theorem]{Algorithm}

\theoremstyle{remark}
\newtheorem{remark}[theorem]{Remark}
\newtheorem{example}[theorem]{Example}
\newtheorem{convention}[theorem]{Convention}

\newcommand{\bA}{\mathbf{A}}
\newcommand{\bC}{\mathbf{C}}
\newcommand{\bN}{\mathbf{N}}
\newcommand{\bP}{\mathbf{P}}
\newcommand{\bQ}{\mathbf{Q}}
\newcommand{\bR}{\mathbf{R}}
\newcommand{\bZ}{\mathbf{Z}}
\newcommand{\cD}{\mathcal{D}}
\newcommand{\cF}{\mathcal{F}}
\newcommand{\cS}{\mathcal{S}}
\newcommand{\cT}{\mathcal{T}}
\newcommand{\fa}{\mathfrak{a}}
\newcommand{\fb}{\mathfrak{b}}
\newcommand{\fc}{\mathfrak{c}}
\newcommand{\fm}{\mathfrak{m}}
\newcommand{\fn}{\mathfrak{n}}
\newcommand{\fp}{\mathfrak{p}}
\newcommand{\sI}{\mathscr{I}}
\newcommand{\sO}{\mathscr{O}}
\newcommand{\sP}{\mathscr{P}}
\newcommand{\sfa}{\mathsf{a}}
\newcommand{\sfb}{\mathsf{b}}
\newcommand{\rd}[1]{{\lfloor{#1}\rfloor}}

\DeclareMathOperator{\id}{id}
\DeclareMathOperator{\mld}{mld}
\DeclareMathOperator{\ord}{ord}
\DeclareMathOperator{\Proj}{Proj}
\DeclareMathOperator{\Spec}{Spec}
\DeclareMathOperator{\wt}{wt}

\begin{document}
\title[On equivalent conjectures on smooth threefolds]{On equivalent conjectures for minimal log discrepancies on smooth threefolds}

\author{Masayuki Kawakita}
\address{Research Institute for Mathematical Sciences, Kyoto University, Kyoto 606-8502, Japan}
\email{masayuki@kurims.kyoto-u.ac.jp}
\thanks{Partially supported by JSPS Grant-in-Aid for Scientific Research (C) 16K05099.}

\begin{abstract}
On smooth threefolds, the ACC for minimal log discrepancies is equivalent to the boundedness of the log discrepancy of some divisor which computes the minimal log discrepancy. We reduce it to the case when the boundary is the product of a canonical part and the maximal ideal to some power. We prove the reduced assertion when the log canonical threshold of the maximal ideal is either at most one-half or at least one.
\end{abstract}

\maketitle

\section{Introduction}
Let $P\in X$ be the germ of a smooth variety and $\fa=\prod_{j=1}^e\fa_j^{r_j}$ be an $\bR$-ideal on $X$. We write $\mld_P(X,\fa)$ for the minimal log discrepancy of the pair $(X,\fa)$ at $P$. For a subset $I$ of the positive real numbers, we mean by $\fa\in I$ that the exponents $r_j$ in $\fa$ belong to $I$. ACC stands for the ascending chain condition while DCC stands for the descending chain condition. This paper discusses the ACC conjecture for minimal log discrepancies on smooth threefolds, which was conjectured by Shokurov \cite{BS10}, \cite{Sh88} for arbitrary lc pairs.

\begin{conjectureAlph}\label{cnj:acc}
Fix a subset $I$ of the positive real numbers which satisfies the DCC. Then the set
\begin{align*}
\{\mld_P(X,\fa)\mid\textup{$P\in X$ a smooth threefold},\ \textup{$\fa$ an $\bR$-ideal},\ \fa\in I\}
\end{align*}
satisfies the ACC.
\end{conjectureAlph}

We approach it with the theory of the generic limit of ideals introduced by de Fernex and Musta\c{t}\u{a} \cite{dFM09}. Our earlier work \cite{K14} shows the finiteness of the set of $\mld_P(X,\fa)$ in which the germ $P\in X$ of a klt variety and the exponents in $\fa$ are fixed. Instead, if the exponents in $\fa$ move in an infinite set satisfying the DCC, then we require the stability of minimal log discrepancies for generic limits. This stability connects Conjecture \ref{cnj:acc} to the following important conjectures equivalently, as it was indicated essentially by Musta\c{t}\u{a} and Nakamura \cite{MN16}.

The first is the ACC for $a$-lc thresholds, a generalisation of lc thresholds.

\begin{conjectureAlph}\label{cnj:alc}
Fix a non-negative real number $a$ and a subset $I$ of the positive real numbers which satisfies the DCC. Then the set
\begin{align*}
\{t\in\bR_{\ge0}\mid\textup{$P\in X$ a smooth threefold},\ \textup{$\fa$, $\fb$ $\bR$-ideals},\ \mld_P(X,\fa\fb^t)=a,\ \fa\fb\in I\}
\end{align*}
satisfies the ACC.
\end{conjectureAlph}

The second is a uniform version of the $\fm$-adic semi-continuity, which was proposed originally by Musta\c{t}\u{a}.

\begin{conjectureAlph}\label{cnj:madic}
Fix a finite subset $I$ of the positive real numbers. Then there exists a positive integer $l$ depending only on $I$ such that if $P\in X$ is the germ of a smooth threefold and if $\fa=\prod_{j=1}^e\fa_j^{r_j}$ and $\fb=\prod_{j=1}^e\fb_j^{r_j}$ are $\bR$-ideals on $X$ satisfying that $r_j\in I$ and $\fa_j+\fm^l=\fb_j+\fm^l$ for any $j$, where $\fm$ is the maximal ideal in $\sO_X$ defining $P$, then $\mld_P(X,\fa)=\mld_P(X,\fb)$.
\end{conjectureAlph}

The last is the boundedness of the log discrepancy of some divisor which computes the minimal log discrepancy, proposed by Nakamura.

\begin{conjectureAlph}\label{cnj:nakamura}
Fix a finite subset $I$ of the positive real numbers. Then there exists a positive integer $l$ depending only on $I$ such that if $P\in X$ is the germ of a smooth threefold and $\fa$ is an $\bR$-ideal on $X$ satisfying that $\fa\in I$, then there exists a divisor $E$ over $X$ which computes $\mld_P(X,\fa)$ and satisfies the inequality $a_E(X)\le l$.
\end{conjectureAlph}

The first main result of this paper is to reduce these conjectures to the case when the boundary is the product of a canonical part and the maximal ideal to some power.

\begin{theorem}\label{thm:first}
Conjectures \textup{\ref{cnj:acc}}, \textup{\ref{cnj:alc}}, \textup{\ref{cnj:madic}} and \textup{\ref{cnj:nakamura}} are equivalent to Conjecture \textup{\ref{cnj:product}}.
\end{theorem}

\begin{conjecture}\label{cnj:product}
Let $P\in X$ be the germ of a smooth threefold and $\fm$ be the maximal ideal in $\sO_X$ defining $P$. Fix a positive rational number $q$ and a non-negative rational number $s$. Then there exists a positive integer $l$ depending only on $q$ and $s$ such that if $\fa$ is an ideal on $X$ satisfying that $(X,\fa^q)$ is canonical, then there exists a divisor $E$ over $X$ which computes $\mld_P(X,\fa^q\fm^s)$ and satisfies the inequality $a_E(X)\le l$.
\end{conjecture}

Our earlier work derives this boundedness when $(X,\fa^q)$ is terminal or $s$ is zero.

\begin{theorem}\label{thm:terminal}
Let $P\in X$ be the germ of a smooth threefold and $\fm$ be the maximal ideal in $\sO_X$ defining $P$.
\begin{enumerate}
\item\label{itm:terminal}
Fix a positive rational number $q$ and a non-negative rational number $s$. Then there exists a positive integer $l$ depending only on $q$ and $s$ such that if $\fa$ is an ideal on $X$ satisfying that $(X,\fa^q)$ is terminal, then there exists a divisor $E$ over $X$ which computes $\mld_P(X,\fa^q\fm^s)$ and satisfies the inequality $a_E(X)\le l$.
\item\label{itm:zero}
Fix a positive rational number $q$. Then there exists a positive integer $l$ depending only on $q$ such that if $\fa$ is an ideal on $X$ satisfying that $(X,\fa^q)$ is canonical, then there exists a divisor $E$ over $X$ which computes $\mld_P(X,\fa^q)$ and satisfies the inequality $a_E(X)\le l$.
\end{enumerate}
\end{theorem}

The second main result is to prove Conjecture \ref{cnj:product} when the lc threshold of the maximal ideal is either at most one-half or at least one.

\begin{theorem}\label{thm:second}
Let $P\in X$ be the germ of a smooth threefold and $\fm$ be the maximal ideal in $\sO_X$ defining $P$. Fix a positive rational number $q$ and a non-negative rational number $s$.
\begin{enumerate}
\item\label{itm:half}
There exists a positive integer $l$ depending only on $q$ and $s$ such that if $\fa$ is an ideal on $X$ satisfying that $(X,\fa^q)$ is canonical and that $\mld_P(X,\fa^q\fm^{1/2})$ is not positive, then there exists a divisor $E$ over $X$ which computes $\mld_P(X,\fa^q\fm^s)$ and satisfies the inequality $a_E(X)\le l$.
\item\label{itm:one}
There exists a positive integer $l$ depending only on $q$ and $s$ such that if $\fa$ is an ideal on $X$ satisfying that $(X,\fa^q)$ is canonical and that $(X,\fa^q\fm)$ is lc, then there exists a divisor $E$ over $X$ which computes $\mld_P(X,\fa^q\fm^s)$ and satisfies the inequality $a_E(X)\le l$.
\end{enumerate}
\end{theorem}

Once Theorem \ref{thm:second}(\ref{itm:half}) is established, it is relatively simple to obtain Conjecture \ref{cnj:product} when $s$ is close to zero in terms of a scale determined by $q$.

\begin{corollary}\label{crl:main}
Conjecture \textup{\ref{cnj:product}} holds when $s$ is at most $1/n$ for some integer $n$ greater than one such that $nq$ is integral.
\end{corollary}

We shall explain the outline of our research. Fix the germ $P\in X$ of a smooth threefold and a positive rational number $q$. For a sequence $\{\fa_i\}_{i\in\bN}$ of ideals on $X$, its generic limit $\sfa$ is defined on the spectrum $\hat P\in \hat X$ of the completion of the local ring $\sO_{X,P}\otimes_kK$, where $K$ is an extension of the ground field $k$. The stability of minimal log discrepancies means the equality
\begin{align*}
\mld_{\hat P}(\hat X,\sfa^q)=\mld_P(X,\fa_i^q)
\end{align*}
for infinitely many $i$, to which any of Conjectures \ref{cnj:acc} to \ref{cnj:nakamura} is equivalent (Theorem \ref{thm:equiv}). The strategy employed in this paper is to pursue Conjecture \ref{cnj:nakamura}.

Our previous work \cite{K15} derived the above stability except for the case when $(\hat X,\sfa^q)$ has the smallest lc centre of dimension one, which implies that $\mld_{\hat P}(\hat X,\sfa^q)$ is at most one. By this result, in order to prove Conjecture \ref{cnj:nakamura}, one has only to consider those ideals $\fa$ which have $\mld_P(X,\fa^q)$ less than one. We begin with the ACC for $1$-lc thresholds \cite{St11}. Using it together with the classification of divisorial contractions \cite{K01}, \cite{Km96}, we construct a birational morphism $Y\to X$ with bounded log discrepancies by which $(X,\fa)$ can be replaced with a pair $(Y,(\fa')^q\fb^q)$ satisfying that $(Y,(\fa')^q)$ is canonical and that $\fb$ has bounded colength (Theorem \ref{thm:canonical}).

We study the generic limit $\sfa$ of a sequence of ideals $\fa_i$ on $X$ such that $(X,\fa_i^q)$ is canonical. We may assume that $(\hat X,\sfa^q)$ has the smallest lc centre $\hat C$ of dimension one. Then the $\mld_{\hat P}(\hat X,\sfa^q)$ equals one by the canonicity of $(X,\fa_i^q)$, and so does the $\mld_P(X,\fa_i^q)$. By our result \cite{K17} in dimension two, there exists a divisor $\hat E$ over $\hat X$ computing $\mld_{\eta_{\hat C}}(\hat X,\sfa^q)=0$ which is obtained at the generic point $\eta_{\hat C}$ of $\hat C$ by a weighted blow-up. On our extra condition that $\mld_{\hat P}(\hat X,\sfa^q)=1$, we find $\hat E$ for which the weighted blow-up at $\eta_{\hat C}$ is extended to the closed point $\hat P$ (Theorem \ref{thm:wbu}).

We associate the minimal log discrepancy on $\hat X$ with that on $\hat E$ by precise inversion of adjunction (Section \ref{sct:reduction}). The generic limit $\sfb$ of a sequence of ideals of bounded colength satisfies that $\sfb\sO_{\hat C}=\hat\fm^b\sO_{\hat C}$ for some integer $b$, where $\hat\fm$ is the maximal ideal in $\sO_{\hat X}$. Then Conjecture \ref{cnj:nakamura} is reduced to the case when $\fb$ is the maximal ideal $\fm$ to the power of $b$, which completes Theorem \ref{thm:first}.

Suppose that the lc threshold of $\fm$ with respect to $(X,\fa^q)$ is at most one-half. Under our assumptions on the generic limit $\sfa$ involved, we derive a special conclusion that $\mld_P(X,\fa^q\fm^s)=1-2s$, which includes the boundedness stated in Theorem \ref{thm:second}(\ref{itm:half}). A similar argument is applied to the case of lc threshold at least one, Theorem \ref{thm:second}(\ref{itm:one}).

Conjecture \ref{cnj:product} remains open when $\mld_P(X,\fa^q)$ equals one and $\mld_P(X,\fa^q\fm^{1/2})$ is positive. In this case, every divisor $E$ over $X$ computing $\mld_P(X,\fa^q)$ satisfies that $\ord_E\fm$ equals one. We supply a classification of the centre of $E$ on a certain weighted blow-up of $X$ (Theorem \ref{thm:crepant}).

\section{Preliminaries}
We shall fix the notation and review the basics of singularities in birational geometry. Refer to \cite{Ko13} for details.

We work over an algebraically closed field $k$ of characteristic zero. We omit to write the bases of tensor products over $k$ and of products over $\Spec k$ when it is clear. A \textit{variety} is an integral separated scheme of finite type over $\Spec k$. The \textit{dimension} of a scheme means the Krull dimension. A variety of dimension one (resp.\ two) is called a \textit{curve} (resp.\ a \textit{surface}).

The \textit{germ} is considered at a closed point algebraically. When we work on the spectrum of a noetherian ring, we identify an ideal in the ring with its coherent ideal sheaf. For an irreducible closed subset $Z$ of a scheme, we write $\eta_Z$ for the generic point of $Z$.

The \textit{round-down} $\rd{r}$ of a real number $r$ is the greatest integer at most $r$. The \textit{natural number} starts from zero.

\medskip
\textit{Orders}.
Let $X$ be a noetherian scheme and $Z$ be an irreducible closed subset of $X$. The \textit{order} of a coherent ideal sheaf $\fa$ on $X$ along $Z$ is the maximal $\nu\in\bN\cup\{+\infty\}$ satisfying that $\fa\sO_{X,\eta_Z}\subset\sI^\nu\sO_{X,\eta_Z}$ for the ideal sheaf $\sI$ of $Z$, and it is denoted by $\ord_Z\fa$. If $Y\to X$ is a birational morphism from a noetherian normal scheme, then we set $\ord_E\fa=\ord_E\fa\sO_Y$ for a prime divisor $E$ on $Y$. The $\ord_Zf$ for a function $f$ in $\sO_X$ stands for $\ord_Z(f\sO_X)$.

Suppose that $X$ is normal. For an effective $\bQ$-Cartier divisor $D$ on $X$, we set $\ord_ZD=r^{-1}\ord_Z\sO_X(-rD)$ for a positive integer $r$ such that $rD$ is Cartier, which is independent of the choice of $r$. The notion of $\ord_ZD$ is extended to $\bR$-Cartier $\bR$-divisors by linearity.

\medskip
\textit{$\bR$-ideals}.
An $\bR$-\textit{ideal} on a noetherian scheme $X$ is a formal product $\fa=\prod_j\fa_j^{r_j}$ of finitely many coherent ideal sheaves $\fa_j$ on $X$ with positive real exponents $r_j$. For a positive real number $t$, the $\fa$ to the \textit{power} of $t$ is $\fa^t=\prod_j\fa_j^{tr_j}$. The \textit{cosupport} of $\fa$ is the union of the supports of $\sO_X/\fa_j$ for all $j$. The \textit{order} of $\fa$ along an irreducible closed subset $Z$ of $X$ is $\ord_Z\fa=\sum_jr_j\ord_Z\fa_j$. The $\fa$ is said to be \textit{invertible} if all $\fa_j$ are invertible. The \textit{pull-back} of $\fa$ by a morphism $Y\to X$ is $\fa\sO_Y=\prod_j(\fa_j\sO_Y)^{r_j}$. For a subset $I$ of the positive real numbers, we mean by $\fa\in I$ that all exponents $r_j$ belong to $I$.

If $\fa$ is invertible, then the $\bR$-divisor $A=\sum_jr_jA_j$ for which $\fa_j=\sO_X(-A_j)$ is called the $\bR$-divisor \textit{defined by} $\fa$. When we work on the germ $P\in X$, the $\bR$-divisor \textit{defined by a general member in} $\fa$ means an $\bR$-divisor $\sum_jr_j(f_j)$ on $X$ with a general member $f_j$ in $\fa_j$. The $\fa$ is said to be \textit{$\fm$-primary} if all $\fa_j$ are $\fm$-primary, where $\fm$ is the maximal ideal in $\sO_X$ defining $P$.

\begin{convention}
It is sometimes convenient to allow an exponent in an $\bR$-ideal to be zero. We define a coherent ideal sheaf to the power of zero as the structure sheaf.
\end{convention}

\medskip
\textit{The minimal log discrepancy}.
A \textit{subtriple} $(X,\Delta,\fa)$ consists of a normal variety $X$, an $\bR$-divisor $\Delta$ on $X$ such that $K_X+\Delta$ is $\bR$-Cartier, and an $\bR$-ideal $\fa$ on $X$. The $(X,\Delta,\fa)$ is called a \textit{triple} if $\Delta$ is effective. We omit to write $\fa$ or $\Delta$ and call $(X,\Delta)$ or $(X,\fa)$ a (\textit{sub})\textit{pair} when $\fa=\sO_X$ or $\Delta=0$. The $\Delta$ or $\fa$ is called the \textit{boundary} when $(X,\Delta)$ or $(X,\fa)$ is a pair.

A prime divisor $E$ on a normal variety $Y$ equipped with a birational morphism $\pi\colon Y\to X$ is called a divisor \textit{over} $X$, and the closure of the image $\pi(E)$ is called the \textit{centre} of $E$ on $X$ and denoted by $c_X(E)$. We write $\cD_X$ for the set of all divisors over $X$. Two elements in $\cD_X$ are usually identified when they define the same valuation on the function field of $X$. The \textit{log discrepancy} of $E$ with respect to $(X,\Delta,\fa)$ is
\begin{align*}
a_E(X,\Delta,\fa)=1+\ord_EK_{Y/(X,\Delta)}-\ord_E\fa,
\end{align*}
where $K_{Y/(X,\Delta)}=K_Y-\pi^*(K_X+\Delta)$.

Let $Z$ be a closed subvariety of $X$. The \textit{minimal log discrepancy} of $(X,\Delta,\fa)$ at the generic point $\eta_Z$ is
\begin{align*}
\mld_{\eta_Z}(X,\Delta,\fa)=\inf\{a_E(X,\Delta,\fa)\mid E\in\cD_X,\ c_X(E)=Z\},
\end{align*}
It is either a non-negative real number or minus infinity. We say that $E\in\cD_X$ \textit{computes} $\mld_{\eta_Z}(X,\Delta,\fa)$ if $c_X(E)=Z$ and $a_E(X,\Delta,\fa)=\mld_{\eta_Z}(X,\Delta,\fa)$ (or is negative when $\mld_{\eta_Z}(X,\Delta,\fa)=-\infty$). It is often enough to study the case when $Z$ is a closed point by the relation $\mld_{\eta_Z}(X,\Delta,\fa)=\mld_P(X,\Delta,\fa)-\dim Z$ for a general closed point $P$ in $Z$. It is sometimes convenient to use the \textit{minimal log discrepancy} of $(X,\Delta,\fa)$ in a closed subset $W$ of $X$ which is defined by
\begin{align*}
\mld_W(X,\Delta,\fa)=\inf\{a_E(X,\Delta,\fa)\mid E\in\cD_X,\ c_X(E)\subset W\}.
\end{align*}

\medskip
\textit{Singularities}.
The subtriple $(X,\Delta,\fa)$ is said to be \textit{log canonical} (\textit{lc}) (resp.\ \textit{Kawamata log terminal} (\textit{klt})) if $a_E(X,\Delta,\fa)\ge0$ (resp.\ $>0$) for all $E\in\cD_X$. It is said to be \textit{purely log terminal} (\textit{plt}) (resp.\ \textit{canonical}, \textit{terminal}) if $a_E(X,\Delta,\fa)>0$ (resp.\ $\ge1$, $>1$) for all $E\in\cD_X$ exceptional over $X$. For a closed point $P$ in $X$, $(X,\Delta,\fa)$ is lc about $P$ iff $\mld_P(X,\Delta,\fa)$ is not minus infinity. When $(X,\Delta,\fa)$ is lc, the \textit{lc threshold} with respect to $(X,\Delta,\fa)$ of a non-trivial $\bR$-ideal $\fb$ on $X$ is the maximal real number $t$ such that $(X,\Delta,\fa\fb^t)$ is lc.

Let $Y$ be a normal variety birational to $X$. A centre $c_Y(E)$ on $Y$ of $E\in\cD_Y$ such that $a_E(X,\Delta,\fa)\le0$ is called a \textit{non-klt centre} on $Y$ of $(X,\Delta,\fa)$. The union of all non-klt centres on $Y$ is called the \textit{non-klt locus} on $Y$ of $(X,\Delta,\fa)$. When we just say a non-klt centre or the non-klt locus, we mean that it is on $X$.

When $(X,\Delta,\fa)$ is lc, a non-klt centre of $(X,\Delta,\fa)$ is often called an \textit{lc centre}. When we work on the germ of a variety, an lc centre contained in every lc centre is called the \textit{smallest lc centre}. The smallest lc centre exists and it is normal \cite[Theorem 9.1]{F11}.

The \textit{index} of a normal $\bQ$-Gorenstein singularity $P\in X$ is the least positive integer $r$ such that $rK_X$ is Cartier at $P$.

\medskip
\textit{Birational transformations}.
A reduced divisor $D$ on a smooth variety $X$ is said to be \textit{simple normal crossing} (\textit{snc}) if $D$ is defined at every closed point $P$ in $X$ by the product of a part of a regular system of parameters in $\sO_{X,P}$. A \textit{stratum} of $D=\sum_{i\in I}D_i$ is an irreducible component of $\bigcap_{i\in I'}D_i$ for a subset $I'$ of $I$. For a smooth morphism $X\to S$, the $D$ said to be snc \textit{relative to} $S$ if every stratum of $D$ is smooth over $S$.

A \textit{log resolution} of a subtriple $(X,\Delta,\fa)$ is a projective birational morphism from a smooth variety $Y$ to $X$ such that
\begin{itemize}
\item
the exceptional locus is a divisor and $\fa\sO_Y$ is invertible,
\item
the union of the exceptional locus, the support of the strict transform of $\Delta$, and the cosupport of $\fa\sO_Y$ is snc, and
\item
it is isomorphic on the maximal open locus $U$ in $X$ such that $U$ is smooth, $\fa\sO_U$ is invertible, and the union of the support of $\Delta|_U$ and the cosupport of $\fa\sO_U$ is snc.
\end{itemize}

Let $(X,\Delta,\fa)$ be a subtriple, where $\fa=\prod_j\fa_j^{r_j}$, and $Y$ be a normal variety birational to $X$. A subtriple $(Y,\Gamma,\fb)$ is said to be \textit{crepant} to $(X,\Delta,\fa)$ if $a_E(X,\Delta,\fa)=a_E(Y,\Gamma,\fb)$ for any divisor $E$ over $X$ and $Y$. Suppose that $Y$ is smooth and has a birational morphism to $X$ whose exceptional locus is a divisor $\sum_iE_i$. The \textit{weak transform} on $Y$ of $\fa$ is the $\bR$-ideal $\fa_Y=\prod_j(\fa_{jY})^{r_j}$ defined by
\begin{align*}
\fa_{jY}=\fa_j\sO_Y(\textstyle\sum_i(\ord_{E_i}\fa_j)E_i).
\end{align*}
Remark that this notion is different from that of the strict transform, the $j$-th ideal of which is $\sum_{f\in\fa_j}f\sO_Y(\sum_i(\ord_{E_i}f)E_i)$ (see \cite[III Definition 5]{H64}). The definition of the weak transform $\fa_Y$ is extended to the case when $Y$ is normal as far as $\sum_i(\ord_{E_i}\fa_j)E_i$ is Cartier for any $j$. We introduce

\begin{definition}
The \textit{pull-back} of $(X,\Delta,\fa)$ by $Y\to X$ is the subtriple $(Y,\Delta_Y,\fa_Y)$ in which $\Delta_Y=-K_{Y/(X,\Delta)}+\sum_{ij}(r_j\ord_{E_i}\fa_j)E_i$.
\end{definition}

The pull-back $(Y,\Delta_Y,\fa_Y)$ is crepant to $(X,\Delta,\fa)$.

\medskip
\textit{Weighted blow-ups}.
Let $P\in X$ be the germ of a smooth variety. Let $x_1,\ldots,x_c$ be a part of a regular system of parameters in $\sO_{X,P}$ and $w_1,\ldots,w_c$ be positive integers. For $w\in\bN$, let $\sI_w$ be the ideal in $\sO_X$ generated by all monomials $x_1^{s_1}\cdots x_c^{s_c}$ such that $\sum_{i=1}^cs_iw_i\ge w$. The \textit{weighted blow-up} of $X$ with $\wt(x_1,\ldots,x_c)=(w_1,\ldots,w_c)$ is $\Proj_X(\bigoplus_{w\in\bN}\sI_w)$.

\begin{remark}\label{rmk:wbu}
If $x'_1,\ldots,x'_c$ is a part of another regular system of parameters such that $x'_i\in\sI_{w_i}\setminus\sI_{w_i+1}$ for any $i$, then the weighted blow-up of $X$ with $\wt(x'_1,\ldots,x'_c)=(w_1,\ldots,w_c)$ is the same that is obtained by $\wt(x_1,\ldots,x_c)=(w_1,\ldots,w_c)$.
\end{remark}

Its explicit description is reduced to the case of the affine space by an \'etale morphism. Let $o\in\bA^d$ be the germ at origin of the affine space with coordinates $x_1,\cdots,x_d$ and $Y$ be the weighted blow-up of $\bA^d$ with $\wt(x_1,\ldots,x_d)=(w_1,\ldots,w_d)$. One may assume that $w_1,\ldots,w_d$ have no common divisors. Then $Y$ is covered by the affine charts $U_i=\bA^d/\bZ_{w_i}(w_1,\ldots,w_{i-1},-1,w_{i+1},\ldots,w_d)$ for $1\le i\le d$, and the exceptional divisor is isomorphic to the weighted projective space $\bP(w_1,\ldots,w_d)$ (see \cite[6.38]{KSC04} for details).

Here the notation $\bA^d/\bZ_r(a_1,\ldots,a_d)$ stands for the quotient of $\bA^d$ by the cyclic group $\bZ_r$ of order $r$ whose generator sends the $i$-th coordinate $x_i$ of $\bA^d$ to $\zeta^{a_i}x_i$, where $\zeta$ is a primitive $r$-th root of unity. The $x_1,\ldots,x_d$ on this quotient are called \textit{orbifold coordinates}. An isolated \textit{cyclic quotient singularity} means that the spectrum of the completion of its local ring coincides with the regular base change of some $\bA^d/\bZ_r(a_1,\ldots,a_d)$, in which it is said to be of \textit{type} $\frac{1}{r}(a_1,\ldots,a_d)$.

In terms of toric geometry following the notation in \cite{I14}, by setting $N=\bZ^d+\bZ v$ where $v=\frac{1}{r}(a_1,\ldots,a_d)$, the quotient $\bA^d/\bZ_r(a_1,\ldots,a_d)$ is the toric variety $T_N(\Delta)$ which corresponds to the cone $\Delta$ spanned by the standard basis $e_1,\ldots,e_d$ of $\bZ^d$. For $e=\frac{1}{r}(w_1,\ldots,w_d)\in N\cap\Delta$, the \textit{weighted blow-up} of $\bA^d/\bZ_r(a_1,\ldots,a_d)$ with respect to $\wt(x_1,\ldots,x_d)=\frac{1}{r}(w_1,\ldots,w_d)$ is defined by adding the ray generated by $e$.

\medskip
\textit{Adjunction}.
Let $X$ be a normal variety and $S+B$ be an effective $\bR$-divisor on $X$ such that $S$ is reduced and has no common components with the support of $B$. Suppose that they form a pair $(X,S+B)$. Then one has the \textit{adjunction}
\begin{align*}
\nu^*(K_X+S+B|_S)=K_{S^\nu}+B_{S^\nu}
\end{align*}
on the normalisation $\nu\colon S^\nu\to S$ of $S$, in which $B_{S^\nu}$ is an effective $\bR$-divisor called the \textit{different} on $S^\nu$ of $B$ (see \cite[Chapter 16]{Ko92} or \cite[Section 3]{Sh93}).

\begin{example}\label{exl:different}
Let $X=\bA^2/\bZ_r(1,w)$ with orbifold coordinates $x_1,x_2$ such that $w$ is coprime to $r$. Let $S$ be the curve on $X$ defined by $x_1$ and $P$ be the origin of $X$. Then $K_X+S|_S=K_S+(1-r^{-1})P$.
\end{example}

The singularity on $X$ is associated with that on $S^\nu$ by

\begin{theorem}[Inversion of adjunction]\label{thm:ia}
Notation as above.
\begin{enumerate}
\item
\textup{(\cite[Theorem 17.6]{Ko92})}\;
$(X,S+B)$ is plt about $S$ iff $(S^\nu,B_{S^\nu})$ is klt. In this case, $S$ is normal.
\item
\textup{(\cite{K07})}\;
$(X,S+B)$ is lc about $S$ iff $(S^\nu,B_{S^\nu})$ is lc.
\end{enumerate}
\end{theorem}

\medskip
\textit{$R$-varieties}.
The notions explained above make sense over the ring $R$ of formal power series over a field of characteristic zero, which has been discussed by de Fernex, Ein and Musta\c{t}\u{a} \cite{dFEM11}, \cite{dFM09}. We mean by an \textit{$R$-variety} an integral separated scheme of finite type over $\Spec R$. We consider regular $R$-varieties instead of smooth $R$-varieties.

The canonical divisor $K_X$ on a normal $R$-variety $X$ is defined by the \textit{sheaf of special differentials} in \cite{dFEM11}. Let $Y\to X$ be a birational morphism between regular $R$-varieties. The relative canonical divisor $K_{Y/X}$ is the effective divisor defined by the zeroth Fitting ideal of $\Omega_{Y/X}$ \cite[Remark A.12]{dFEM11}. In particular, $K_{Y/X}$ is independent of the structure of $X$ as an $R$-variety.

\begin{remark}\label{rmk:regular}
Let $P\in X$ be the germ of a normal $\bQ$-Gorenstein $R$-variety and $\fa$ be an $\bR$-ideal on $X$. Let $X'$ be either
\begin{itemize}
\item
the spectrum of the completion of the local ring $\sO_{X,P}$, or
\item
$X\times_{\Spec R}\Spec R'$, where $R'$ is the completion of $R\otimes_KK'$ for a field extension $K'$ of $K$,
\end{itemize}
which has a regular morphism $\pi\colon X'\to X$. Then $K_{X'}=\pi^*K_X$, by which one has that $a_{E'}(X',\fa\sO_{X'})=a_E(X,\fa)$ and $\mld_{P'}(X',\fa\sO_{X'})=\mld_P(X,\fa)$ for any components $E'$ of $E\times_XX'$ and $P'$ of $P\times_XX'$.
\end{remark}

\begin{lemma}\label{lem:regular}
Let $P\in X$ be the germ of an $R$-variety. Let $\hat X$ be the spectrum of the completion of the local ring $\sO_{X,P}$ and $\hat P$ be its closed point.
\begin{enumerate}
\item\label{itm:idealbij}
Let $\fm$ be the maximal ideal in $\sO_X$ defining $P$ and $\hat\fm$ be the maximal ideal in $\sO_{\hat X}$. Then the pull-back defines a bijective map from the set of $\fm$-primary $\bR$-ideals on $X$ to the set of $\hat\fm$-primary $\bR$-ideals on $\hat X$.
\item\label{itm:divisorbij}
Suppose that $X$ is normal. Then the base change defines a bijective map from the set of divisors over $X$ with centre $P$ to the set of divisors over $\hat X$ with centre $\hat P$.
\end{enumerate}
\end{lemma}

\begin{proof}
The (\ref{itm:idealbij}) follows from the isomorphisms $\sO_X/\fm^l\simeq\sO_{\hat X}/\hat\fm^l$, while (\ref{itm:divisorbij}) follows from the property that blowing-up commutes with flat base changes.
\end{proof}

By Lemma \ref{lem:regular}, in order to study the minimal log discrepancy at the closed point of the germ $P\in X$, one may often replace $P\in X$ with a germ the completion of the local ring of which is isomorphic to that of $\sO_{X,P}$.

\section{The generic limit of ideals}\label{sct:limit}
We recall the generic limit of ideals on a fixed germ. It was introduced by de Fernex and Musta\c{t}\u{a} \cite{dFM09} and simplified by Koll\'ar \cite{Ko08}. We follow our style of the definition in \cite{K15}.

Let $P\in X$ be the germ of a scheme of finite type over $k$ and $\fm$ be the maximal ideal in $\sO_X$ defining $P$. Let $\cS=\{(\fa_{i1},\ldots,\fa_{ie})\}_{i\in\bN}$ be an infinite sequence of $e$-tuples of ideals in $\sO_X$. For every positive integer $l$, the ideal $(\fa_{ij}+\fm^l)/\fm^l$ in $\sO_X/\fm^l$ for $i\in\bN$ and $1\le j\le e$ corresponds to a closed point $P_{ij}(l)$ in the Hilbert scheme $H_l$ parametrising ideals in $\sO_X/\fm^l$. The $H_l$ is a scheme of finite type over $k$ and there exists a natural rational map $H_{l+1}\to H_l$. Take the Zariski closure $Z_j(l)$ of the subset $\{P_{ij}(l)\}_{i\in\bN}$ in $H_l$. By finding a locally closed irreducible subset $Z_l$ of $Z_1(l)\times\cdots\times Z_e(l)$ inductively, one obtains a family of approximations of $\cS$ defined below.

\begin{definition}\label{dfn:approx}
A \textit{family} $\cF=(Z_l,(\fa_j(l))_j,N_l,s_l,t_l)_{l\ge l_0}$ \textit{of approximations} of $\cS$ consists of a fixed positive integer $l_0$ and for every $l\ge l_0$,
\begin{itemize}
\item
a variety $Z_l$,
\item
an ideal sheaf $\fa_j(l)$ on $X\times Z_l$ for every $1\le j\le e$ which is flat over $Z_l$ and contains $\fm^l\sO_{X\times Z_l}$,
\item
an infinite subset $N_l$ of $\bN$ and a map $s_l\colon N_l\to Z_l(k)$, where $Z_l(k)$ denotes the set of the $k$-points in $Z_l$, and
\item
a dominant morphism $t_l\colon Z_{l+1}\to Z_l$,
\end{itemize}
such that
\begin{itemize}
\item
$\fa_j(l)\sO_{X\times Z_{l+1}}=\fa_j(l+1)+\fm^l\sO_{X\times Z_{l+1}}$ by $\id_X\times t_l$,
\item
$\fa_j(l)_i=\fa_{ij}+\fm^l$ for $i\in N_l$, where $\fa_j(l)_i=\fa_j(l)\otimes_{\sO_{Z_l}}k$ is the ideal in $\sO_X$ given by the closed point $s_l(i)\in Z_l$,
\item
the image of $N_l$ by $s_l$ is dense in $Z_l$, and
\item
$N_{l+1}$ is contained in $N_l$ and $t_l\circ s_{l+1}=s_l|_{N_{l+1}}$.
\end{itemize}
\end{definition}

For the above $\cF$, let $K=\varinjlim_lK(Z_l)$ be the union of the function fields $K(Z_l)$ of $Z_l$ by the inclusions $t_l^*\colon K(Z_l)\to K(Z_{l+1})$. Let $\hat X$ be the spectrum of the completion of the local ring $\sO_{X,P}\otimes_kK$. Let $\hat P$ be the closed point of $\hat X$ and $\hat\fm$ be the maximal ideal in $\sO_{\hat X}$.

\begin{definition}
The \textit{generic limit} of $\cS$ with respect to $\cF$ is the $e$-tuple $(\sfa_1,\ldots,\sfa_e)$ of ideals in $\sO_{\hat X}$ defined by
\begin{align*}
\sfa_j=\varprojlim_l\fa_j(l)_K,
\end{align*}
where $\fa_j(l)_K=\fa_j(l)\otimes_{\sO_{Z_l}}K$ is the ideal in $\sO_X\otimes_kK$ given by the natural $K$-point $\Spec K\to Z_l$.
\end{definition}

\begin{remark}
Let $R$ be the completion of the local ring $\sO_{X,P}$. In the literature, the generic limit is defined for a sequence of $e$-tuples of ideals in $R$. When $\fa_{ij}$ are ideals in $R$, the generic limit is defined in the same way just by replacing the condition $\fa_j(l)_i=\fa_{ij}+\fm^l$ in Definition \ref{dfn:approx} with $\fa_j(l)_i=(\fa_{ij}+\fm^lR)\cap\sO_X$.
\end{remark}

By the very definition, one has

\begin{lemma}
Let $(\sfa,\sfb)$ be a generic limit of a sequence $\{(\fa_i,\fb_i)\}_{i\in\bN}$ of pairs of ideals in $\sO_X$.
\begin{enumerate}
\item
If $\fa_i\subset\fb_i$ for any $i$, then $\sfa\subset\sfb$.
\item
The $\sfa+\sfb$ and $\sfa\sfb$ are the generic limits of $\{\fa_i+\fb_i\}_{i\in\bN}$ and $\{\fa_i\fb_i\}_{i\in\bN}$.
\end{enumerate}
\end{lemma}

The generic limit depends on the choice of $\cF$ but remains the same after the replacement of $\cF$ with a subfamily.

\begin{definition}
A family $\cF'=(Z'_l,(\fa'_j(l))_j,N'_l,s'_l,t'_l)_{l\ge l'_0}$ of approximations of $\cS$ is called a \textit{subfamily} of $\cF$ if $l'_0$ is at least $l_0$ and if there exists an open immersion $i_l\colon Z'_l\to Z_l$ for every $l\ge l'_0$ such that
\begin{itemize}
\item
$t_l\circ i_{l+1}=i_l\circ t'_l$,
\item
$\fa_j(l)\sO_{X\times Z'_l}=\fa'_j(l)$ by $\id_X\times i_l$, and
\item
$N'_l$ is a subset of $N_l$ and $i_l\circ s'_l=s_l|_{N'_l}$.
\end{itemize}
\end{definition}

\begin{convention}\label{cnv:retain}
Later we shall often replace $\cF$ with a subfamily, but we retain the same notation $\cF=(Z_l,(\fa_j(l))_j,N_l,s_l,t_l)_{l\ge l_0}$ to avoid intricacy.
\end{convention}

The theory of the generic limit of ideals was developed for the study of the singularities on the germ $P\in X$. When $X$ is klt, the singularities on $\hat X$ are associated with those on $X$ (see \cite{dFEM11}). The existence of log resolutions supplies

\begin{lemma}\label{lem:resolution}
Notation as above and assume that $X$ is klt. Then $\hat X$ is klt, and after replacing $\cF$ with a subfamily but using the same notation,
\begin{align*}
\mld_{\hat P}(\hat X,{\textstyle\prod}_{j=1}^e(\sfa_j+\hat\fm^l)^{r_j})=\mld_P(X,{\textstyle\prod}_{j=1}^e\fa_j(l)_i^{r_j})
\end{align*}
for any positive real numbers $r_1,\ldots,r_e$ and for any $i\in N_l$ and $l\ge l_0$.
\end{lemma}

\begin{remark}\label{rmk:descend}
\begin{enumerate}
{\first\item\label{itm:descendE}}
Let $\hat E$ be a divisor over $\hat X$ with centre $\hat P$. Then replacing $\cF$ with a subfamily (but using the same notation as in Convention \ref{cnv:retain}), one can descend $\hat E$ to a divisor $E_l$ over $X\times Z_l$ for any $l\ge l_0$, that is, $E_{l'}=E_l\times_{Z_l}Z_{l'}$ when $l\le l'$, and $\hat E=E_l\times_{X\times Z_l}\hat X$. Let $E_i$ be any connected component of the fibre of $E_l$ at $s_l(i)\in Z_l$, which is independent of $l$ as far as $i\in N_l$. Replacing $\cF$ with a subfamily again, for any $i\in N_l$ and $1\le j\le e$, $E_i$ is a divisor over $X$ and satisfies that
\begin{align*}
\ord_{\hat E}\sfa_j=\ord_{\hat E}(\sfa_j+\hat\fm^l)&=\ord_{E_i}(\fa_{ij}+\fm^l)=\ord_{E_i}\fa_{ij}<l,\\
a_{\hat E}(\hat X,\sfa)&=a_{E_i}(X,\fa_i).
\end{align*}
\item
Let $\hat\pi\colon\hat Y\to\hat X$ be a projective birational morphism isomorphic outside $\hat P$. Then $\hat\pi$ is descendible as stated in \cite[Proposition A.7]{K15}, that is, after replacing $\cF$ with a subfamily, there exist projective morphisms $\pi_l\colon Y_l\to X\times Z_l$ such that $\pi_{l'}=\pi_l\times_{Z_l}Z_{l'}$ when $l\le l'$ and such that $\hat\pi=\pi_l\times_{X\times Z_l}\hat X$.
\end{enumerate}
\end{remark}

\begin{remark}
The $E_l$ is treated in \cite{K15} as if it has connected fibres, which should have been corrected appropriately.
\end{remark}

Now we fix positive real numbers $r_1,\ldots,r_e$ and consider the $\bR$-ideals $\fa_i=\prod_{j=1}^e\fa_{ij}^{r_j}$ and $\sfa=\prod_{j=1}^e\sfa_j^{r_j}$. The $\sfa$ is called the \textit{generic limit} of the sequence $\{\fa_i\}_{i\in\bN}$ of $\bR$-ideals on $X$ with respect to $\cF$. The most important achievement at present is the following theorem due to de Fernex, Ein and Musta\c{t}\u{a}. Indeed, as an application, they proved first the ACC for lc thresholds restricted on smooth varieties.

\begin{theorem}[\cite{dFEM10}, \cite{dFEM11}]\label{thm:lct}
Notation as above and assume that $X$ is klt. If $(\hat X,\sfa)$ is lc, then so is $(X,\fa_i)$ for any $i\in N_{l_0}$ after replacing $\cF$ with a subfamily which depends on $r_1,\ldots,r_e$.
\end{theorem}

Theorem \ref{thm:lct} is a corollary to the effective $\fm$-adic semi-continuity of lc thresholds, which was globalised in \cite[Theorem 4.11]{K15}. We prove its relative version.

\begin{theorem}\label{thm:relative}
Let $X$ be a klt variety and $X\to T$ be a morphism to a variety. Suppose that every closed fibre of $X\to T$ is klt. Let $\fa=\prod_j\fa_j^{r_j}$ be an $\bR$-ideal on $X$ and $Z$ be an irreducible closed subset of $X$ which dominates $T$. Suppose that $\mld_{\eta_Z}(X,\fa)=0$ and it is computed by a divisor $E$ over $X$. Then after replacing $X$ and $T$ with their dense open subsets, the following hold for any $t\in T$.
\begin{itemize}
\item
The fibre of $E$ at $t$ is non-empty, and its arbitrary connected component $E_t$ is a divisor over a component $X_t$ of the fibre of $X$ at $t$.
\item
The centre $Z_t$ on $X_t$ of $E_t$ is smooth.
\item
If an $\bR$-ideal $\fb=\prod_j\fb_j^{r_j}$ on $X_t$ satisfies that $\fa_j\sO_{X_t}+\fp_j=\fb_j+\fp_j$ for any $j$, where $\fp_j=\{f\in\sO_{X_t}\mid\ord_{E_t}f>\ord_E\fa_j\}$, then $(X_t,\fb)$ is lc about $Z_t$ and $\mld_{\eta_{Z_t}}(X_t,\fb)=0$.
\end{itemize}
\end{theorem}

\begin{proof}
Take a log resolution $\pi\colon Y\to X$ of $(X,\fa\sI_Z)$, where $\sI_Z$ is the ideal sheaf of $Z$, such that $E$ is realised as a divisor on $Y$. We may shrink $T$ so that $T$ and $Y\to T$ are smooth and so that the union $F$ of the exceptional locus of $\pi$ and the cosupport of $\fa\sI_Z\sO_Y$ is an snc divisor relative to $T$. Replace $X$ with an open subset $X'$ containing $\eta_Z$ such that $Z'=Z|_{X'}$ is smooth over $T$ and such that if the restriction $S'=S|_{\pi^{-1}(X')}$ of a stratum $S$ of $F$ satisfies that $S'\neq\emptyset$ and $\pi(S')\subset Z'$, then $S'\to Z'$ is smooth.

Set $n=\dim Z-\dim T$. Then for any $t\in T$ and $z\in Z_t$,
\begin{align*}
\mld_z(X_t,\fm_z^n\cdot\fa\sO_{X_t})=0
\end{align*}
for the maximal ideal sheaf $\fm_z$ on $X_t$ defining $z$, and it is computed by the divisor $G_z$ obtained by the blow-up of $Y_t=Y\times_XX_t$ along a component of $E_t\cap\pi^{-1}(z)$. This is verified from the local description at each closed point $y$ in $\pi^{-1}(z)$. Indeed, let $v_1,\ldots,v_s$ be a part of a regular system of parameters in $\sO_{Y,y}$ such that $F$ is defined at $y$ by $\prod_{l=1}^sv_l$. Since every stratum of $F$ mapped into $Z$ is smooth over $Z$, they are extended to a part $v_1,\ldots,v_s,w_1,\ldots,w_n$ of a regular system of parameters in $\sO_{Y,y}$ such that their images form a part of a regular system of parameters in $\sO_{Y_t,y}$ and such that
\begin{align*}
\fm_z\sO_{Y_t,y}=(w_1,\ldots,w_n,\textstyle\prod_{l=1}^sv_l^{m_l})\sO_{Y_t,y},
\end{align*}
where $m_l$ is the order of $\sI_Z$ along the divisor defined by $v_l$. (Note that the corresponding expression in the proof of \cite[Theorem 4.11]{K15} is incorrect).

Since $\ord_{G_z}\fa_j\sO_{X_t}=\ord_E\fa_j$ and $\ord_{G_z}f\ge\ord_{E_t}f$ for any $f\in\sO_{X_t}$, by \cite[Theorem 1.4]{dFEM10} we conclude that $\mld_z(X_t,\fm_z^n\fb)=0$ for the $\fb$ in the statement. Hence $(X_t,\fb)$ is lc about $Z_t$, and $\mld_{\eta_{Z_t}}(X_t,\fb)=0$ by $a_{E_t}(X_t,\fb)=0$.
\end{proof}

\begin{corollary}\label{crl:relative}
Let $X$ be a klt variety and $X\to T$ be a morphism to a variety. Suppose that the fibre $X_t$ at every closed point $t$ in $T$ is klt. Let $\fa=\prod_j\fa_j^{r_j}$ be an $\bR$-ideal on $X$ and $Z$ be a closed subset of $X$ such that $(X,\fa)$ is lc about $Z$. Set $Z_t=Z\times_XX_t$ and let $\sI_t$ denote the ideal sheaf of $Z_t$ on $X_t$. Then there exists a positive integer $l$ such that after replacing $T$ with its dense open subset, for any $t\in T$ if an $\bR$-ideal $\fb=\prod_j\fb_j^{r_j}$ on $X_t$ satisfies that $\fa_j\sO_{X_t}+\sI_t^l=\fb_j+\sI_t^l$ for any $j$, then $(X_t,\fb)$ is lc about $Z_t$.
\end{corollary}

\begin{proof}
We shall prove it by noetherian induction on $Z$. Let $Z_0$ be an irreducible component of $Z$, which may be assumed to dominate $T$. Let $\sI_{Z_0}$ be the ideal sheaf of $Z_0$ and $r$ be the non-negative real number such that $\mld_{\eta_{Z_0}}(X,\fa\sI_{Z_0}^r)$ equals zero. Applying Theorem \ref{thm:relative}, after shrinking $T$ there exist open subset $X'$ of $X$ containing $\eta_{Z_0}$ and a positive integer $l_0$ such that for any $t\in T$, if an $\bR$-ideal $\fb=\prod_j\fb_j^{r_j}$ on $X_t$ satisfies that $\fa_j\sO_{X'_t}+\sI_{Z_0}^{l_0}\sO_{X'_t}=\fb_j\sO_{X'_t}+\sI_{Z_0}^{l_0}\sO_{X'_t}$ for any $j$ on $X'_t=X'\times_XX_t$, then $(X'_t,\fb\sO_{X'_t})$ is lc about $Z_0\times_XX'_t$. Thus the assertion is reduced to that for the closure of $Z\setminus(Z_0\cap X')$, which follows from the hypothesis of induction.
\end{proof}

\section{Singularities on a fixed variety}
In this section, we fix the germ $P\in X$ of a klt variety and review an approach to the study of $\mld_P(X,\fa)$ for $\bR$-ideals $\fa$ which uses the generic limit of ideals on $X$. Our earlier work shows the discreteness for log discrepancies $a_E(X,\fa)$.

\begin{theorem}[\cite{K14}]\label{thm:discrete}
Let $P\in X$ be the germ of a klt variety. Fix a finite subset $I$ of the positive real numbers. Then the set
\begin{align*}
\{a_E(X,\fa)\mid\textup{$\fa$ an $\bR$-ideal},\ \fa\in I,\ E\in\cD_X,\ \textrm{$(X,\fa)$ lc about $\eta_{c_X(E)}$}\}
\end{align*}
is discrete in $\bR$.
\end{theorem}

We shall explain the equivalence of several important conjectures on a fixed germ with the help of Theorems \ref{thm:lct} and \ref{thm:discrete}.

\begin{conjecture}\label{cnj:equiv}
Let $P\in X$ be the germ of a klt variety.
\begin{enumerate}[series=equiv]
\item\label{itm:acc}
\textup{(ACC for minimal log discrepancies)}\;
Fix a subset $I$ of the positive real numbers which satisfies the DCC. Then the set
\begin{align*}
\{\mld_P(X,\fa)\mid\textup{$\fa$ an $\bR$-ideal},\ \fa\in I\}
\end{align*}
satisfies the ACC.
\item\label{itm:alc}
\textup{(ACC for $a$-lc thresholds)}\;
Fix a non-negative real number $a$ and a subset $I$ of the positive real numbers which satisfies the DCC. Then the set
\begin{align*}
\{t\in\bR_{\ge0}\mid\textup{$\fa$, $\fb$ $\bR$-ideals},\ \mld_P(X,\fa\fb^t)=a,\ \fa\fb\in I\}
\end{align*}
satisfies the ACC.
\item\label{itm:madic}
\textup{(uniform $\fm$-adic semi-continuity)}\;
Fix a finite subset $I$ of the positive real numbers. Then there exists a positive integer $l$ depending only on $X$ and $I$ such that if $\fa=\prod_{j=1}^e\fa_j^{r_j}$ and $\fb=\prod_{j=1}^e\fb_j^{r_j}$ are $\bR$-ideals on $X$ satisfying that $r_j\in I$ and $\fa_j+\fm^l=\fb_j+\fm^l$ for any $j$, where $\fm$ is the maximal ideal in $\sO_X$ defining $P$, then $\mld_P(X,\fa)=\mld_P(X,\fb)$.
\item\label{itm:nakamura}
\textup{(boundedness)}\;
Fix a finite subset $I$ of the positive real numbers. Then there exists a positive integer $l$ depending only on $X$ and $I$ such that if $\fa$ is an $\bR$-ideal on $X$ satisfying that $\fa\in I$, then there exists a divisor $E$ over $X$ which computes $\mld_P(X,\fa)$ and satisfies the inequality $a_E(X)\le l$.
\item\label{itm:limit}
\textup{(generic limit)}\;
Let $r_1,\ldots,r_e$ be positive real numbers and $\{\fa_i=\prod_{j=1}^e\fa_{ij}^{r_j}\}_{i\in\bN}$ be a sequence of $\bR$-ideals on $X$. Notation as in Section \textup{\ref{sct:limit}}, so set the generic limit $\sfa=\prod_{j=1}^e\sfa_j^{r_j}$ on $\hat P\in\hat X$. Then $\mld_{\hat P}(\hat X,\sfa)=\mld_P(X,\fa_i)$ for any $i\in N_{l_0}$ after replacing $\cF$ with a subfamily but using the same notation.
\end{enumerate}
\end{conjecture}

\begin{remark}\label{rmk:limit}
We provide a few remarks on Conjecture \ref{cnj:equiv}(\ref{itm:limit}).
\begin{enumerate}
\item\label{itm:limitineq}
Lemma \ref{lem:resolution} means the equality
\begin{align*}
\mld_{\hat P}(\hat X,{\textstyle\prod}_{j=1}^e(\sfa_j+\hat\fm^l)^{r_j})=\mld_P(X,{\textstyle\prod}_{j=1}^e(\fa_{ij}+\fm^l)^{r_j})
\end{align*}
for any $i\in N_{l_0}$ after replacing $\cF$ with a subfamily. Take a divisor $\hat E$ over $\hat X$ which computes $\mld_{\hat P}(\hat X,\sfa)$ and choose an integer $l_1\ge l_0$ such that $\ord_{\hat E}\sfa_j\le l_1\ord_{\hat E}\hat\fm$ for any $j$. Then for $l\ge l_1$, the left-hand side equals $\mld_{\hat P}(\hat X,\sfa)$ while the right-hand side is at least $\mld_P(X,\fa_i)$. Thus after replacing $l_0$ with $l_1$, one has the inequality
\begin{align*}
\mld_{\hat P}(\hat X,\sfa)\ge\mld_P(X,\fa_i)
\end{align*}
for any $i\in N_{l_0}$. The intrinsic part in Conjecture \ref{cnj:equiv}(\ref{itm:limit}) is the opposite inequality.
\item\label{itm:limitresult}
In particular, Conjecture \ref{cnj:equiv}(\ref{itm:limit}) holds when $\mld_{\hat P}(\hat X,\sfa)$ is not positive by Theorem \ref{thm:lct}. The conjecture also holds when $(\hat X,\sfa)$ is klt \cite[Theorem 5.1]{K15}. Thus, we know that Conjecture \ref{cnj:equiv}(\ref{itm:limit}) holds unless $(\hat X,\sfa)$ is not klt but $\mld_{\hat P}(\hat X,\sfa)$ is positive.
\end{enumerate}
\end{remark}

We prepare basic lemmata.

\begin{lemma}\label{lem:DCC}
Let $I$ and $J$ be subsets of the positive real numbers both of which satisfy the DCC. Then the set $\{rs\mid r\in I,\ s\in J\}$ satisfies the DCC.
\end{lemma}

\begin{proof}
Let $\{r_is_i\}_{i\in\bN}$ be an arbitrary non-increasing sequence where $r_i\in I$ and $s_i\in J$. It is enough to show that $r_is_i$ is constant passing to a subsequence. We claim that there exists a strictly increasing sequence $\{i_j\}_{j\in\bN}$ such that $\{r_{i_j}\}_{j\in\bN}$ is a non-decreasing sequence. Indeed, let $i_1$ be a number such that $r_{i_1}$ attains the minimum of the set $\{r_i\mid i\in\bN\}$, which exists since this set satisfies the DCC. If one constructed $i_1,\ldots,i_j$, then take $i_{j+1}$ as a number such that $r_{i_{j+1}}$ attains the minimum of the set $\{r_i\mid i>i_j\}$.

By replacing $\{r_is_i\}_{i\in\bN}$ with $\{r_{i_j}s_{i_j}\}_{j\in\bN}$, we may assume that $r_i$ is non-decreasing. Applying the same argument to $\{s_i\}_{i\in\bN}$, we may also assume that $s_i$ is non-decreasing. Then the sequence $\{r_is_i\}_{i\in\bN}$ becomes both non-increasing and non-decreasing, so $r_is_i$ must be constant.
\end{proof}

\begin{lemma}\label{lem:mld}
Let $P\in X$ be the germ of a normal $\bQ$-Gorenstein variety and $\fa_1,\ldots,\fa_e$ be $\bR$-ideals on $X$. Let $t_1,\ldots,t_e$ be non-negative real numbers such that $\sum_{i=1}^et_i=1$.
\begin{enumerate}
\item\label{itm:mldconvex}
$\mld_P(X,\prod_{i=1}^e\fa_i^{t_i})\ge\sum_{i=1}^et_i\mld_P(X,\fa_i)$.
\item\label{itm:mldequal}
If a divisor $E$ over $X$ computes all $\mld_P(X,\fa_i)$, then $\mld_P(X,\prod_{i=1}^e\fa_i^{t_i})=\sum_{i=1}^et_i\mld_P(X,\fa_i)$ and it is computed by $E$.
\end{enumerate}
\end{lemma}

\begin{proof}
Let $F$ be a divisor over $X$ which computes $\mld_P(X,\prod_{i=1}^e\fa_i^{t_i})$. Then,
\begin{align*}
\mld_P(X,{\textstyle\prod_{i=1}^e\fa_i^{t_i}})=a_F(X,{\textstyle\prod_{i=1}^e\fa_i^{t_i}})=\sum_{i=1}^et_i\cdot a_F(X,\fa_i)\ge\sum_{i=1}^et_i\mld_P(X,\fa_i),
\end{align*}
which is (\ref{itm:mldconvex}). On the other hand, if $E$ computes all $\mld_P(X,\fa_i)$, then
\begin{align*}
\mld_P(X,{\textstyle\prod_{i=1}^e\fa_i^{t_i}})\le a_E(X,{\textstyle\prod_{i=1}^e\fa_i^{t_i}})=\sum_{i=1}^et_i\cdot a_E(X,\fa_i)=\sum_{i=1}^et_i\mld_P(X,\fa_i),
\end{align*}
which with (\ref{itm:mldconvex}) shows the assertion (\ref{itm:mldequal}).
\end{proof}

\begin{theorem}\label{thm:equiv}
Let $P\in X$ be the germ of a klt variety. Then the five statements in Conjecture \textup{\ref{cnj:equiv}} are equivalent.
\end{theorem}

\begin{proof}
\textit{Step} 1.
The generic limit of ideals was invented from the insight of the implication from (\ref{itm:limit}) to (\ref{itm:acc}). Musta\c{t}\u{a} informed us the proof of this implication and we wrote it in \cite[Proposition 4.8]{K15}. Note that the proof in \cite{K15} works even if $X$ has klt singularities. We also note that though the statement in \cite{K15} assumes the assertion in (\ref{itm:limit}) for ideals $\fa_{ij}$ in the completion of the local ring $\sO_{X,P}$, its proof uses only the assertion for ideals in $\sO_X$ which is exactly (\ref{itm:limit}). We derived from (\ref{itm:limit}) in fact the following ACC which was formulated by Cascini and McKernan \cite{M13}.
\begin{enumerate}[topsep=\smallskipamount,resume=equiv]
\item\label{itm:CM}
\textit{Fix subsets $I$ of the positive real numbers and $J$ of the non-negative real numbers both of which satisfy the DCC. Then there exist finite subsets $I_0$ of $I$ and $J_0$ of $J$ such that if $\fa$ is an $\bR$-ideal on $X$ satisfying that $\fa\in I$ and $\mld_P(X,\fa)\in J$, then $\fa\in I_0$ and $\mld_P(X,\fa)\in J_0$.}
\end{enumerate}

The assertion (\ref{itm:acc}) follows from (\ref{itm:CM}) immediately. We shall derive (\ref{itm:alc}) from (\ref{itm:CM}). Let $\{t_i\}_{i\in\bN}$ be a non-decreasing sequence of positive real numbers such that there exist ideals $\fa_i$ and $\fb_i$ on $X$ satisfying that $\mld_P(X,\fa_i\fb_i^{t_i})=a$ and $\fa_i\fb_i\in I$. It is enough to show that $T=\{t_i\mid i\in\bN\}$ satisfies the ACC. By Lemma \ref{lem:DCC}, the set $IT=\{rt\mid r\in I,\ t\in T\}$ satisfies the DCC. Applying (\ref{itm:CM}) to $I\cup IT$ and $\{a\}$, one obtains a finite subset $I_0$ of $I\cup IT$ such that $\fa_i\fb_i^{t_i}\in I_0$ for any $i$. Particularly, $T$ is contained in the set $I^{-1}I_0=\{r^{-1}s\mid r\in I,\ s\in I_0\}$ which satisfies the ACC.

\medskip
\textit{Step} 2.
The conjecture (\ref{itm:nakamura}) was proposed by Nakamura. His joint work \cite{MN16} with Musta\c{t}\u{a} shows the equivalence of (\ref{itm:madic}), (\ref{itm:nakamura}) and (\ref{itm:limit}). They treated the assertion in (\ref{itm:limit}) for ideals in the completion, but their proof works for our (\ref{itm:limit}). They also provided a direct proof of the implication from (\ref{itm:nakamura}) to (\ref{itm:acc}) which uses the ACC for lc thresholds on $X$ and Theorem \ref{thm:discrete}. We write the argument from (\ref{itm:limit}) to (\ref{itm:nakamura}) as Lemma \ref{lem:limtonak} since it will be used later.

\medskip
\textit{Step} 3.
Hence it is enough to show the implications from (\ref{itm:acc}) to (\ref{itm:nakamura}) and from (\ref{itm:alc}) to (\ref{itm:nakamura}). If (\ref{itm:nakamura}) were false, then there would exist a strictly increasing sequence $\{l_i\}_{i\in\bN}$ and a sequence $\{\fa_i\}_{i\in\bN}$ of $\bR$-ideals on $X$ such that $\fa_i\in I$ and such that every divisor $E_i$ over $X$ computing $\mld_P(X,\fa_i)$ satisfies the inequality $a_{E_i}(X)\ge l_i$. The assertion (\ref{itm:nakamura}) for those $\fa$ whose $\mld_P(X,\fa)$ is not positive will be proved in Theorem \ref{thm:nonpos} independently. We assume that $\mld_P(X,\fa_i)$ is positive for any $i$ here.

By Theorem \ref{thm:discrete}, the set
\begin{align*}
M=\{a_E(X,\fa)\mid\fa\in I,\ E\in\cD_X,\ \textrm{$(X,\fa)$ lc}\}
\end{align*}
is discrete in $\bR$. In particular, all $\mld_P(X,\fa_i)$ belong to a finite set since they are bounded from above by $\mld_PX$. Thus we may assume that $\mld_P(X,\fa_i)$ is constant, say $m$, which is positive by our assumption. We may assume that $\fa_i$ is non-trivial, then $m$ is less than $\mld_PX$. By the discreteness of $M$, there exists a real number $m'$ greater than $m$ such that $r\not\in M$ for any real number $m<r\le m'$.

Let $t_i$ be the positive real number such that $\mld_P(X,\fa_i^{1-t_i})=m'$, which exists and satisfies that $0<t_i<1$ by $m<m'<\mld_PX$. Take a divisor $E_i$ over $X$ which computes $\mld_P(X,\fa_i^{1-t_i})$. Then $a_{E_i}(X,\fa_i)<m'$, so $E_i$ also computes $\mld_P(X,\fa_i)=m$ by the property of $m'$, and thus $\ord_{E_i}\fa_i=a_{E_i}(X)-m\ge l_i-m$. Since $t_i\ord_{E_i}\fa_i=a_{E_i}(X,\fa_i^{1-t_i})-a_{E_i}(X,\fa_i)=m'-m$, one has the estimate $t_i\le(m'-m)/(l_i-m)$ when $l_i>m$, showing that $t_i$ approaches to zero as $i$ increases.

This contradicts the ACC for $m'$-lc thresholds in (\ref{itm:alc}). It is sufficient to verify that our situation also contradicts the ACC for minimal log discrepancies in (\ref{itm:acc}). By passing to a subsequence, we may assume that $t_i$ are less than one-half and form a strictly decreasing sequence whose limit is zero. Then $\{1-(1-t_i)t_i\}_{i\in \bN}$ is a strictly increasing sequence. We set
\begin{align*}
T=\{1-(1-t_i)t_i\mid i\in\bN\},
\end{align*}
which satisfies the DCC.

Note that $1-(1-t_i)t_i=(1-t_i)(1-t_i)+t_i$. Because $E_i$ computes both $\mld_P(X,\fa_i^{1-t_i})$ and $\mld_P(X,\fa_i)$, by Lemma \ref{lem:mld}(\ref{itm:mldequal}) one has that
\begin{align*}
\mld_P(X,\fa_i^{1-(1-t_i)t_i})=(1-t_i)\mld_P(X,\fa_i^{1-t_i})+t_i\mld_P(X,\fa_i)=m'-t_i(m'-m)
\end{align*}
which is computed by $E_i$. But then $\mld_P(X,\fa_i^{1-(1-t_i)t_i})$ is strictly increasing. This contradicts (\ref{itm:acc}) for $IT=\{rt\mid r\in I,\ t\in T\}$ since $IT$ satisfies the DCC by Lemma \ref{lem:DCC}.
\end{proof}

\begin{lemma}\label{lem:limtonak}
Let $P\in X$ be the germ of a klt variety. Let $r_1,\ldots,r_e$ be positive real numbers and $\{\fa_i=\prod_{j=1}^e\fa_{ij}^{r_j}\}_{i\in\bN}$ be a sequence of $\bR$-ideals on $X$. Notation as in Section \textup{\ref{sct:limit}}, so set the generic limit $\sfa=\prod_{j=1}^e\sfa_j^{r_j}$ on $\hat P\in\hat X$. If $\mld_{\hat P}(\hat X,\sfa)=\mld_P(X,\fa_i)$ for any $i\in N_{l_0}$, then there exists a positive rational number $l$ such that for infinitely many indices $i$, there exists a divisor $E_i$ over $X$ which computes $\mld_P(X,\fa_i)$ and satisfies the equality $a_{E_i}(X)=l$.
\end{lemma}

\begin{proof}
Take a divisor $\hat E$ over $\hat X$ which computes $\mld_{\hat P}(\hat X,\sfa)$. As in Remark \ref{rmk:descend}(\ref{itm:descendE}), replacing $\cF$ with a subfamily, one can descend $\hat E$ to a divisor $E_l$ over $X\times Z_l$ for any $l\ge l_0$. For a component $E_i$ of the fibre of $E_l$ at $s_l(i)\in Z_l$, one may assume that $a_{\hat E}(\hat X,\sfa)=a_{E_i}(X,\fa_i)$ for any $i\in N_{l_0}$. Then $E_i$ computes $\mld_P(X,\fa_i)$ and $a_{E_i}(X)$ equals the constant $a_{\hat E}(\hat X)$.
\end{proof}

\begin{theorem}\label{thm:nonpos}
Let $P\in X$ be the germ of a klt variety. Fix a finite subset $I$ of the positive real numbers. Then there exists a positive integer $l$ depending only on $X$ and $I$ such that if $\fa$ is an $\bR$-ideal on $X$ satisfying that $\fa\in I$ and that $\mld_P(X,\fa)$ is not positive, then there exists a divisor $E$ over $X$ which computes $\mld_P(X,\fa)$ and satisfies the inequality $a_E(X)\le l$.
\end{theorem}

\begin{proof}
Let $\{\fa_i\}_{i\in\bN}$ be an arbitrary sequence of $\bR$-ideals on $X$ such that $\fa_i\in I$ and such that $\mld_P(X,\fa_i)$ is not positive. It is sufficient to show the existence of a positive rational number $l$ such that for infinitely many indices $i$, there exists a divisor $E_i$ over $X$ which computes $\mld_P(X,\fa_i)$ and satisfies the equality $a_{E_i}(X)=l$.

Write $\fa_i=\prod_{j=1}^{e_i}\fa_{ij}^{r_{ij}}$ so $r_{ij}\in I$. We may assume that every $\fa_{ij}$ is non-trivial. Let $r$ be the minimum of the elements of $I$. Then $\mld_P(X,\fa_i)\le\mld_PX-\sum_{j=1}^{e_i}r_{ij}\le\mld_PX-re_i$. Let $e'$ denote the greatest integer such that $re'\le\mld_PX$. If $(X,\fa_i)$ is lc, then $e_i\le e'$. If $(X,\fa_i)$ is not lc and $e_i>e'$, then we may replace $\fa_i$ with $\fa'_i=\prod_{j=1}^{e'+1}\fa_{ij}^{r_{ij}}$ because every divisor computing $\mld_P(X,\fa'_i)=-\infty$ also computes $\mld_P(X,\fa_i)$. Hence by passing to a subsequence, we may assume that $e_i$ is constant, say $e$, and that $r_{ij}$ is constant, say $r_j$, for each $1\le j\le e$. That is, $\fa_i=\prod_{j=1}^e\fa_{ij}^{r_j}$.

Following Section \ref{sct:limit}, we construct a generic limit $\sfa=\prod_{j=1}^e\sfa_j^{r_j}$ of $\{\fa_i\}_{i\in\bN}$. We use the notation in Section \ref{sct:limit}, so $\sfa$ is an $\bR$-ideal on $\hat P\in\hat X$. If $\mld_{\hat P}(\hat X,\sfa)$ were positive, then there would exist a positive real number $t$ such that $(\hat X,\sfa\hat\fm^t)$ is lc. By Theorem \ref{thm:lct}, $(X,\fa_i\fm^t)$ is lc for infinitely many $i$, which contradicts that $\mld_P(X,\fa_i)$ is not positive. Thus $\mld_{\hat P}(\hat X,\sfa)$ is not positive. Then by Remark \ref{rmk:limit}(\ref{itm:limitresult}), $\mld_{\hat P}(\hat X,\sfa)=\mld_P(X,\fa_i)$ for any $i\in N_{l_0}$ after replacing $\cF$ with a subfamily, and the existence of $l$ follows from Lemma \ref{lem:limtonak}.
\end{proof}

The conjectures hold in dimension two.

\begin{theorem}\label{thm:surface}
Conjecture \textup{\ref{cnj:equiv}} holds when $X$ is a klt surface.
\end{theorem}

\begin{proof}
By Theorem \ref{thm:equiv}, it is enough to verify one of the statements. The (\ref{itm:nakamura}) is stated in \cite[Theorem 1.3]{MN16}. Alternatively, one may derive (\ref{itm:acc}) from \cite[Theorem 3.8]{Al94}, or derive (\ref{itm:limit}) from \cite{K13} by replacing $X$ with its minimal resolution.
\end{proof}

Roughly speaking, our former work \cite{K15} asserts a part of the conjectures in dimension three in the case when the minimal log discrepancy is greater than one.

\begin{theorem}\label{thm:grthan1}
Let $P\in X$ be the germ of a smooth threefold. Let $r_1,\ldots,r_e$ be positive real numbers and $\{\fa_i=\prod_{j=1}^e\fa_{ij}^{r_j}\}_{i\in\bN}$ be a sequence of $\bR$-ideals on $X$. Notation as in Section \textup{\ref{sct:limit}}, so set the generic limit $\sfa=\prod_{j=1}^e\sfa_j^{r_j}$ on $\hat P\in\hat X$. Then the pair $(\hat X,\sfa)$ satisfies one of the following cases.
\begin{enumerate}[label=\textup{\arabic*.},ref=\arabic*]
\item\label{cas:case1}
The $\mld_{\hat P}(\hat X,\sfa)$ is not positive.
\item\label{cas:case2}
$(\hat X,\sfa)$ is klt.
\item\label{cas:case3}
$(\hat X,\sfa)$ is lc and has the smallest lc centre which is normal and of dimension two.
\item\label{cas:case4}
$(\hat X,\sfa)$ is lc and has the smallest lc centre which is regular and of dimension one.
\end{enumerate}

Moreover, the following hold.
\begin{enumerate}
\item\label{itm:cases123}
In the cases \textup{\ref{cas:case1}}, \textup{\ref{cas:case2}} and \textup{\ref{cas:case3}}, $\mld_{\hat P}(\hat X,\sfa)=\mld_P(X,\fa_i)$ for any $i\in N_{l_0}$ after replacing $\cF$ with a subfamily.
\item\label{itm:case4}
In the case \textup{\ref{cas:case4}}, $\mld_{\hat P}(\hat X,\sfa)$ is at most one.
\end{enumerate}
\end{theorem}

\begin{proof}
The case division follows from the existence of the smallest lc centre \cite[Theorem 1.2]{K15}. The equality in (\ref{itm:cases123}) holds in the cases \textup{\ref{cas:case1}} and \textup{\ref{cas:case2}} by Remark \ref{rmk:limit}(\ref{itm:limitresult}), and in the case \textup{\ref{cas:case3}} by \cite[Theorem 5.3]{K15}. The assertion (\ref{itm:case4}) is \cite[Proposition 6.1]{K15}.
\end{proof}

We reduce Conjecture \ref{cnj:equiv}(\ref{itm:nakamura}) to the case of $\bQ$-ideals.

\begin{lemma}\label{lem:rational}
Let $P\in X$ be the germ of a klt variety. Suppose that for any positive integer $n$, there exists a positive integer $l$ depending only on $X$ and $n$ such that if $\fa$ is an ideal on $X$, then there exists a divisor $E$ over $X$ which computes $\mld_P(X,\fa^{1/n})$ and satisfies the inequality $a_E(X)\le l$. Then Conjecture \textup{\ref{cnj:equiv}} holds for $P\in X$.
\end{lemma}

\begin{proof}
In Conjecture \ref{cnj:equiv}, the assertion (\ref{itm:limit}) follows from (\ref{itm:nakamura}) for $I=\{r_1,\ldots,r_e\}$ by \cite{MN16}. Thus we may assume Conjecture \ref{cnj:equiv}(\ref{itm:limit}) in the case when $e=1$ and $r_1=1/n$ for some positive integer $n$. By Theorem \ref{thm:equiv}, it is enough to derive the full statement of (\ref{itm:limit}) from this special case.

We want the equality $\mld_{\hat P}(\hat X,\sfa)=\mld_P(X,\fa_i)$, where $\fa_i=\prod_{j=1}^e\fa_{ij}^{r_j}$ and $\sfa=\prod_{j=1}^e\sfa_j^{r_j}$. We write $m=\mld_{\hat P}(\hat X,\sfa)$ for simplicity. By Remark \ref{rmk:limit}(\ref{itm:limitresult}), we may assume that $m$ is positive. By Theorem \ref{thm:discrete}, the set
\begin{align*}
M=\{\mld_P(X,\fa)\mid\fa\in\{r_1,\ldots,r_e\},\ \textrm{$(X,\fa)$ lc}\}
\end{align*}
is discrete in $\bR$. Thus there exists a real number $m'$ less than $m$ such that $r\not\in M$ for any real number $m'<r<m$.

Since the set
\begin{align*}
Q=\{(q_1,\ldots,q_e)\in(\bR_{\ge0})^e\mid\textrm{$(\hat X,{\textstyle\prod_{j=1}^e}\sfa_j^{q_j})$ lc}\}
\end{align*}
is a rational polytope, the vector $r=(r_1,\ldots,r_e)$ in $Q$ is expressed as $r=\sum_{s\in S}t_sq_s$, where $S$ is a finite set, all $q_s=(q_{1s},\ldots,q_{es})$ belong to $Q\cap\bQ^e$, and $t_s$ are positive real numbers such that $\sum_{s\in S}t_s=1$. By choosing $q_s$ close to $r$, we may assume that
\begin{align*}
m'=\mld_{\hat P}(\hat X,\sfa)-(m-m')<\sum_{s\in S}t_s\mld_{\hat P}(\hat X,{\textstyle\prod_{j=1}^e}\sfa_j^{q_{js}}).
\end{align*}

Write $q_{js}=m_{js}/n$ with positive integers $n$ and $m_{1s},\ldots,m_{es}$ for $s\in S$. Then $\mld_{\hat P}(\hat X,\prod_{j=1}^e\sfa_j^{q_{js}})=\mld_{\hat P}(\hat X,(\prod_{j=1}^e\sfa_j^{m_{js}})^{1/n})$ and the ideal $\prod_{j=1}^e\sfa_j^{m_{js}}$ is the generic limit of the sequence $\{\prod_{j=1}^e\fa_{ij}^{m_{js}}\}_{i\in\bN}$ of ideals on $X$. By our assumption, the equality $\mld_{\hat P}(\hat X,(\prod_{j=1}^e\sfa_j^{m_{js}})^{1/n})=\mld_P(X,(\prod_{j=1}^e\fa_{ij}^{m_{js}})^{1/n})$ holds for any $i\in N_{l_0}$ and $s\in S$ after replacing $\cF$ with a subfamily. Hence with Lemma \ref{lem:mld}(\ref{itm:mldconvex}), one has that
\begin{align*}
m'<\sum_{s\in S}t_s\mld_P(X,{\textstyle\prod_{j=1}^e}\fa_{ij}^{q_{js}})\le\mld_P(X,{\textstyle\prod_{j=1}^e}\fa_{ij}^{r_j})=\mld_P(X,\fa_i)\in M,
\end{align*}
which implies that $\mld_P(X,\fa_i)\ge m$ by the property of $m'$. Together with Remark \ref{rmk:limit}(\ref{itm:limitineq}), we obtain the required equality $m=\mld_P(X,\fa_i)$.
\end{proof}

\begin{proposition}\label{prp:mult}
Let $P\in X$ be the germ of a klt variety and $\fm$ be the maximal ideal in $\sO_X$ defining $P$. Fix a finite subset $I$ of the positive real numbers. Then there exists a positive real number $t$ depending only on $X$ and $I$ such that if $\fa$ is an $\bR$-ideal on $X$ satisfying that $\fa\in I$ and that $\mld_P(X,\fa)$ is positive, then $(X,\fa\fm^t)$ is lc.
\end{proposition}

\begin{proof}
Fix $r_1,\ldots,r_e\in I$ and let $\{\fa_i=\prod_{j=1}^e\fa_{ij}^{r_j}\}_{i\in\bN}$ be a sequence of $\bR$-ideals on $X$ such that $\mld_P(X,\fa_i)$ is positive. It is enough to show the existence of a positive real number $t$ such that $(X,\fa_i\fm^t)$ is lc for infinitely many indices $i$.

Following Section \ref{sct:limit}, we construct a generic limit $\sfa$ of $\{\fa_i\}_{i\in\bN}$ on $\hat P\in\hat X$. Then $\mld_{\hat P}(\hat X,\sfa)$ is positive by Remark \ref{rmk:limit}(\ref{itm:limitineq}), so there exists a positive real number $t$ such that $(\hat X,\sfa\hat\fm^t)$ is lc for the maximal ideal $\hat\fm$ in $\sO_{\hat X}$. By Theorem \ref{thm:lct}, there exists an infinite subset $N_{l_0}$ of $\bN$ such that $(X,\fa_i\fm^t)$ is lc for any $i\in N_{l_0}$.
\end{proof}

\begin{corollary}[\cite{MN16}]\label{crl:mult}
Let $P\in X$ be the germ of a klt variety and $\fm$ be the maximal ideal in $\sO_X$ defining $P$. Fix a finite subset $I$ of the positive real numbers. Then there exists a positive integer $b$ depending only on $X$ and $I$ such that if $\fa$ is an $\bR$-ideal on $X$ satisfying that $\fa\in I$ and that $\mld_P(X,\fa)$ is positive, then $\ord_E\fm$ is at most $b$ for every divisor $E$ over $X$ computing $\mld_P(X,\fa)$.
\end{corollary}

\begin{proof}
Take $t$ in Proposition \ref{prp:mult}. Let $E$ be an arbitrary divisor over $X$ which computes $\mld_P(X,\fa)$. The log canonicity of $(X,\fa\fm^t)$ implies that
\begin{align*}
\ord_E\fm^t\le a_E(X,\fa)=\mld_P(X,\fa)\le\mld_PX,
\end{align*}
that is, $\ord_E\fm\le t^{-1}\mld_PX$. The $b=\rd{t^{-1}\mld_PX}$ is a required integer.
\end{proof}

\section{Construction of canonical pairs}
The objective of this section is to prove the following theorem.

\begin{theorem}\label{thm:canonical}
Let $P\in X$ be the germ of a smooth threefold. Fix a positive rational number $q$. Then there exist positive integers $l$ and $c$ both of which depend only on $q$ such that if $\fa$ is an ideal on $X$ satisfying that $\mld_P(X,\fa^q)$ is positive, then at least one of the following holds.
\begin{enumerate}
\item\label{itm:bounded}
There exists a divisor $E$ over $X$ which computes $\mld_P(X,\fa^q)$ and satisfies the inequality $a_E(X)\le l$.
\item\label{itm:reduced}
There exists a birational morphism from the germ $Q\in Y$ of a smooth threefold to the germ $P\in X$ such that
\begin{itemize}
\item
every exceptional prime divisor $F$ on $Y$ satisfies the inequalities $a_F(X)\le c$ and $a_F(X,\fa^q)<1$, by which the pull-back $(Y,\Delta,\fa_Y^q)$ of $(X,\fa^q)$ is defined with an effective $\bQ$-divisor $\Delta$,
\item
$\mld_Q(Y,\Delta,\fa_Y^q)=\mld_P(X,\fa^q)$, and
\item
$\mld_Q(Y,\fa_Y^q)$ is at least one.
\end{itemize}
\end{enumerate}
\end{theorem}

We recall a part of the classification of threefold divisorial contractions which will play an important role in our argument.

\begin{definition}
A projective birational morphism $Y\to X$ between $\bQ$-factorial terminal varieties is called a \textit{divisorial contraction} if its exceptional locus is a prime divisor.
\end{definition}

\begin{theorem}\label{thm:divisorial}
Let $\pi\colon Y\to X$ be a threefold divisorial contraction which contracts its exceptional divisor to a closed point $P$ in $X$.
\begin{enumerate}
\item\label{itm:kawamata}
\textup{(\cite{Km96})}\;
Suppose that $P$ is a quotient singularity of $X$. The spectrum of the completion of $\sO_{X,P}$ is the regular base change of $\bA^3/\bZ_r(w,-w,1)$ with orbifold coordinates $x_1,x_2,x_3$, where $w$ is a positive integer less than $r$ and coprime to $r$. Then $\pi$ is base-changed to the weighted blow-up with $\wt(x_1,x_2,x_3)=(w/r,(r-w)/r,1/r)$.
\item\label{itm:kawakita}
\textup{(\cite{K01})}\;
Suppose that $P$ is a smooth point of $X$. Then there exists a regular system $x_1,x_2,x_3$ of parameters in $\sO_{X,P}$ and coprime positive integers $w_1,w_2$ such that $\pi$ is the weighted blow-up with $\wt(x_1,x_2,x_3)=(w_1,w_2,1)$.
\end{enumerate}
\end{theorem}

Stepanov proved the ACC for canonical thresholds on smooth threefolds as an application of Theorem \ref{thm:divisorial}(\ref{itm:kawakita}).

\begin{theorem}[\cite{St11}]\label{thm:stepanov}
The set
\begin{align*}
\{t\in\bQ_{\ge0}\mid\textup{$P\in X$ a smooth threefold},\ \textup{$\fa$ an ideal},\ \mld_P(X,\fa^t)=1\}
\end{align*}
satisfies the ACC.
\end{theorem}

\begin{proof}
Let $S$ denote the set in the theorem. The original statement \cite[Theorem 1.7]{St11} asserts that the set
\begin{align*}
T=\biggl\{t\in\bQ_{\ge0}\;\biggm|
\begin{array}{l}
\textrm{$P\in X$ a smooth threefold},\ \textrm{$D$ an effective divisor},\\
\textrm{$(X,tD)$ canonical but not terminal}
\end{array}
\biggr\}
\end{align*}
satisfies the ACC. It is enough to show that if $t$ is an arbitrary element of $S$, then $t/3$ belongs to $T$. For such $t$, there exists an ideal $\fa$ on the germ $P\in X$ of a smooth threefold such that $\mld_P(X,\fa^t)=1$. Then $t\le\mld_PX=3$ and $t/3$ is en element of $S$ since $\mld_P(X,(\fa^3)^{t/3})=1$. Thus it is sufficient to show that if $t\in S$ is at most one, then $t$ belongs to $T$.

Take a germ $P\in X$ on which $t$ is realised by $\mld_P(X,\fa^t)=1$. Let $\fm$ be the maximal ideal in $\sO_X$ defining $P$. By replacing $\fa$ with $\fa+\fm^l$ for a large integer $l$, we may assume that $\fa$ is $\fm$-primary. We take a log resolution $Y$ of $(X,\fa)$. Then $\fa\sO_Y=\sO_Y(-A)$ for an effective divisor $A$ such that $-A$ is free over $X$. Thus there exists a reduced divisor $H$ linearly equivalent to $-A$ such that $Y$ is also a log resolution of $(X,H_X,\fa)$, where $H_X$ is the push-forward of $H$. Then $\mld_P(X,tH_X)=1$ and $(X,tH_X)$ is canonical, meaning that $t\in T$.
\end{proof}

We shall use a consequence of the minimal model program.

\begin{definition}
Let $P\in X$ be the germ of a $\bQ$-factorial terminal variety and $\fa$ be an $\bR$-ideal on $X$ such that $\mld_P(X,\fa)$ equals one. A divisorial contraction to $X$ is said to be \textit{crepant} with respect to $(P,X,\fa)$ if its exceptional divisor computes $\mld_P(X,\fa)$.
\end{definition}

\begin{lemma}\label{lem:crepant}
Let $P\in X$ be the germ of a $\bQ$-factorial terminal variety and $\fa$ be an $\bR$-ideal on $X$ such that $\mld_P(X,\fa)$ equals one. Then there exists a divisorial contraction crepant with respect to $(P,X,\fa)$.
\end{lemma}

\begin{proof}
By replacing $\fa$ with $\fb$ in Lemma \ref{lem:perturb}, we may assume that $\fa$ is an $\fm$-primary $\bR$-ideal, where $\fm$ is the maximal ideal in $\sO_X$ defining $P$, such that there exists a unique divisor $E$ over $X$ which computes $\mld_P(X,\fa)$. Then by \cite[Corollary 1.4.3]{BCHM10}, there exists a projective birational morphism $Y\to X$ from a $\bQ$-factorial normal variety whose exceptional locus is a prime divisor which coincides with $E$. We may assume that the weak transform $\fa_Y$ on $Y$ of $\fa$ is defined. Then $(Y,\fa_Y)$ is the pull-back of $(X,\fa)$, and it is terminal by the uniqueness of $E$. In particular $Y$ itself is terminal, so $Y\to X$ is a required contraction.
\end{proof}

\begin{lemma}\label{lem:perturb}
Let $P\in X$ be the germ of a $\bQ$-factorial klt variety and $\fa$ be an $\bR$-ideal on $X$ such that $(X,\fa)$ is lc. Then there exists an $\bR$-ideal $\fb$ such that
\begin{itemize}
\item
$\fb$ is $\fm$-primary, where $\fm$ is the maximal ideal in $\sO_X$ defining $P$,
\item
$\mld_P(X,\fa)=\mld_P(X,\fb)$,
\item
there exists a unique divisor $E$ over $X$ which computes $\mld_P(X,\fb)$, and
\item
$E$ also computes $\mld_P(X,\fa)$.
\end{itemize}
\end{lemma}

\begin{proof}
Writing $\fa=\prod_j\fa_j^{r_j}$, if we take a large integer $l$, then $\fa'=\prod_j(\fa_j+\fm^l)^{r_j}$ satisfies that $\mld_P(X,\fa')=\mld_P(X,\fa)$ and any divisor computing $\mld_P(X,\fa')$ also computes $\mld_P(X,\fa)$. By replacing $\fa$ with $\fa'$, we may assume that $\fa$ is $\fm$-primary.

Let $Y$ be a log resolution of $(X,\fa)$ and $\{E_i\}_{i\in I}$ be the set of the exceptional prime divisors on $Y$ contracting to the point $P$. Let $A$ be the $\bR$-divisor on $Y$ defined by $\fa\sO_Y$ and $I'$ be the subset of $I$ consisting of the indices $i$ such that $E_i$ computes $\mld_P(X,\fa)$. There exists an effective exceptional divisor $F$ such that $-F$ is very ample and such that the minimum $m$ of $\{\ord_{E_i}A/\ord_{E_i}F\}_{i\in I'}$ attains by only one index, say $i_0\in I'$. One can take a small positive real number $\epsilon$ such that $b_i=a_{E_i}(X,\fa)+\epsilon(\ord_{E_i}A-m\ord_{E_i}F)$ is greater than $\mld_P(X,\fa)$ for any $i\in I\setminus\{i_0\}$. Note that $b_{i_0}$ remains equal to $\mld_P(X,\fa)$.

Let $\fc$ be the ideal on $X$ given by the push-forward of $\sO_Y(-F)$ and set the $\bR$-ideal $\fb=\fa^{1-\epsilon}\fc^{\epsilon m}$. Possibly replacing $\epsilon$ with a smaller real number, we may assume that $(X,\fb)$ is lc. $Y$ is also a log resolution of $(X,\fb)$ and $a_{E_i}(X,\fb)=b_i$ for any $i\in I$. Thus $\fb$ satisfies all the required properties but being $\fm$-primary. However, one can replace $\fb$ with an $\fm$-primary $\bR$-ideal just by the argument of constructing $\fa'$ from $\fa$.
\end{proof}

We consider the following algorithm in order to prove Theorem \ref{thm:canonical}.

\begin{algorithm}\label{alg:canonical}
Let $q$ be a positive rational number. Let $P\in X$ be the germ of a smooth threefold and $\fa$ be an ideal on $X$ such that $(X,\fa^q)$ is lc. Let $E$ be a divisor over $X$ which computes $\mld_P(X,\fa^q)$.
\begin{enumerate}[indented,label=\texttt{\arabic*}.,ref=\texttt{\arabic*}]
\item
Start with $X_0=X$.
\item\label{prc:initial}
Suppose that $X_i$ is given, which has only terminal quotient singularities.
\item\label{prc:notpoint}
If the centre $c_{X_i}(E)$ on $X_i$ of $E$ is of positive dimension, then output $X_i$.
\item\label{prc:point}
Suppose that $c_{X_i}(E)$ is a closed point, which will be denoted by $P_i$. Let $r_i$ be the index of the germ $P_i\in X_i$. One can define the weak transform $\fb_i$ on $P_i\in X_i$ of $\fa^{r_i}$ and let $\fa_i$ be the $\bQ$-ideal $\fb_i^{1/r_i}$. The pair $(X_i,\fa_i^q)$ is lc at $P_i$.
\item\label{prc:smoothout}
If $P_i$ is a smooth point of $X_i$ and $\mld_{P_i}(X_i,\fa_i^q)\ge1$, then output $X_i$.
\item
If $P_i$ is a smooth point of $X_i$ and $\mld_{P_i}(X_i,\fa_i^q)<1$, then go to \ref{prc:iplus1}.
\item\label{prc:singout}
If $P_i$ is a singular point of $X_i$ and $\mld_{P_i}(X_i,\fa_i^q)>1$, then output $X_i$.
\item
If $P_i$ is a singular point of $X_i$ and $\mld_{P_i}(X_i,\fa_i^q)\le1$, then go to \ref{prc:iplus1}.
\item\label{prc:iplus1}
Let $q_i$ be the positive rational number such that $\mld_{P_i}(X_i,\fa_i^{q_i})=1$. Fix a divisorial contraction $X_{i+1}\to X_i$ crepant with respect to $(P_i,X_i,\fa_i^{q_i})$ by Lemma \ref{lem:crepant}. Go back to \ref{prc:initial} and proceed with $X_{i+1}$ instead of $X_i$.
\end{enumerate}
In this algorithm, we fix the notation that $F_i$ is the exceptional divisor of $X_{i+1}\to X_i$ and that $F_j^i$ is the strict transform on $X_i$ of $F_j$ for $j<i$, and we set $\Delta_i=\sum_{j=0}^{i-1}(1-a_{F_j}(X,\fa^q))F_j^i$ and $S_i=\sum_{j=0}^{i-1}F_j^i$.
\end{algorithm}

\begin{remark}
By the very definition,
\begin{enumerate}
\first\item
$q_i\le q$, and $q_i<q$ when $q_i$ is defined at a smooth point $P_i$ of $X_i$, and
\item
$(X_i,\Delta_i,\fa_i^q)$ is crepant to $(X,\fa^q)$.
\end{enumerate}
\end{remark}

In order to run Algorithm \ref{alg:canonical}, we need to verify that
\begin{itemize}
\item
$X_i$ has at worst quotient singularities in the process \ref{prc:initial}, and
\item
$(X_i,\fa_i^q)$ is lc at $P_i$ in the process \ref{prc:point},
\end{itemize}
besides the termination. The first claim follows from Theorem \ref{thm:divisorial}. For the second claim, let $b_i=1-a_{F_i}(X_i,\fa_i^q)$. Then $(X_{i+1},b_iF_i,\fa_{i+1}^q)$ is crepant to $(X_i,\fa_i^q)$ and $b_iF_i$ is effective since $a_{F_i}(X_i,\fa_i^q)\le a_{F_i}(X_i,\fa_i^{q_i})=1$ by $q_i\le q$. Thus the log canonicity of $(X_i,\fa_i^q)$ follows from that of $(X,\fa^q)$ inductively. Hence Algorithm \ref{alg:canonical} runs up to the termination.

We prepare several basic properties of the algorithm before completing its termination in Proposition \ref{prp:termination}.

\begin{lemma}\label{lem:algorithm}
The following hold in Algorithm \textup{\ref{alg:canonical}}.
\begin{enumerate}
\item\label{itm:nondecr}
The $q_i$ form a non-decreasing sequence.
\item\label{itm:atmost1}
$a_{F_i}(X,\fa^{q_i})\le1$.
\item\label{itm:lessthan1}
$a_{F_i}(X,\fa^q)<1$.
\item\label{itm:easybound}
If $q_i\neq q$, then $a_{F_i}(X)\le q(q-q_i)^{-1}$.
\end{enumerate}
\end{lemma}

\begin{proof}
Since $(X_{i+1},\fa_{i+1}^{q_i})$ is crepant to $(X_i,\fa_i^{q_i})$, one has that $\mld_{P_{i+1}}(X_{i+1},\fa_{i+1}^{q_i})\ge\mld_{P_i}(X_i,\fa_i^{q_i})=1$. Thus $q_i\le q_{i+1}$, which shows (\ref{itm:nondecr}).

For $j<i$, let $b_{ij}=1-a_{F_j}(X_j,\fa_j^{q_i})$. Then $(X_{j+1},b_{ij}F_j,\fa_{j+1}^{q_i})$ is crepant to $(X_j,\fa_j^{q_i})$ and $b_{ij}F_j$ is effective by $q_j\le q_i$ in (\ref{itm:nondecr}). Thus one has the inequality $a_{F_i}(X_j,\fa_j^{q_i})\le a_{F_i}(X_{j+1},\fa_{j+1}^{q_i})$ and inductively $a_{F_i}(X,\fa^{q_i})\le a_{F_i}(X_i,\fa_i^{q_i})=1$, which is (\ref{itm:atmost1}).

The (\ref{itm:lessthan1}) follows from (\ref{itm:atmost1}) unless $q_i=q$. If $q_i=q$, then $q_i$ is defined at a singular point $P_i$ of $X_i$, so $q_0<q$. Let $j$ be the greatest integer such that $q_j<q$. Then $(X_i,\fa_i^q)$ is crepant to $(X_{j+1},\fa_{j+1}^q)$. On the other hand, $(X_{j+1},\Delta_{j+1},\fa_{j+1}^q)$ is crepant to $(X,\fa^q)$. By (\ref{itm:atmost1}), $\Delta_{j+1}$ is effective and its support coincides with $S_{j+1}$. Thus one has that $a_{F_i}(X,\fa^q)=a_{F_i}(X_{j+1},\Delta_{j+1},\fa_{j+1}^q)<a_{F_i}(X_{j+1},\fa_{j+1}^q)=a_{F_i}(X_i,\fa_i^q)=1$.

To see (\ref{itm:easybound}), suppose that $q_i<q$. By (\ref{itm:atmost1}) and $a_{F_i}(X,\fa^q)\ge0$, one computes that
\begin{align*}
\fa_{F_i}(X)=a_{F_i}(X,\fa^{q_i})+q_i\ord_{F_i}\fa&\le1+q_i\ord_{F_i}\fa\\
&=1+q_i(q-q_i)^{-1}(a_{F_i}(X,\fa^{q_i})-a_{F_i}(X,\fa^q))\\
&\le1+q_i(q-q_i)^{-1}=q(q-q_i)^{-1}.
\end{align*}
\end{proof}

\begin{lemma}\label{lem:e}
Fix a positive rational number $q$. Then there exists a positive rational number $\epsilon$ depending only on $q$ such that every $q_i$ defined at a smooth point $P_i$ of $X_i$ in Algorithm \textup{\ref{alg:canonical}} satisfies the inequality $q_i\le q-\epsilon$.
\end{lemma}

\begin{proof}
It follows from Theorem \ref{thm:stepanov}.
\end{proof}

\begin{proposition}\label{prp:termination}
Algorithm \textup{\ref{alg:canonical}} terminates.
\end{proposition}

\begin{proof}
Take $\epsilon$ in Lemma \ref{lem:e}. Then every divisor $F_i$ defined over a smooth point $P_i$ of $X_i$ satisfies that $a_{F_i}(X,\fa^{q-\epsilon})\le a_{F_i}(X,\fa^{q_i})\le1$ by Lemma \ref{lem:algorithm}(\ref{itm:atmost1}). The number of such $F_i$ is finite because $(X,\fa^{q-\epsilon})$ is klt. In particular, there exists an integer $e$, depending on $(X,\fa^q)$ and $E$, such that $P_i$ is a singular point of $X_i$ for any $i>e$. By Theorem \ref{thm:divisorial}(\ref{itm:kawamata}), the $r_i$ for $i>e$ form a strictly decreasing sequence. Hence the algorithm must terminate.
\end{proof}

One can also bound $r_i$ and $a_{F_i}(X)$.

\begin{lemma}\label{lem:r}
Fix a positive rational number $q$. Then there exists a positive integer $r$ depending only on $q$ such that every $r_i$ in Algorithm \textup{\ref{alg:canonical}} satisfies the inequality $r_i\le r$.
\end{lemma}

\begin{proof}
Take $\epsilon$ in Lemma \ref{lem:e}. We shall show that any positive integer $r$ at least $q\epsilon^{-1}-2$ satisfies the required property. The $r_0$ is one. By Theorem \ref{thm:divisorial}, the $r_{i+1}$ satisfies that $r_{i+1}<r_i$ when $P_i$ is a singular point of $X_i$ and that $r_{i+1}\le a_{F_i}(X_i)-2$ when $P_i$ is a smooth point of $X_i$. Thus it is enough to show that $a_{F_i}(X_i)\le q\epsilon^{-1}$ when $P_i$ is a smooth point of $X_i$. Since
\begin{align*}
a_{F_i}(X_i)\le a_{F_i}(X_i)+\sum_{j=0}^{i-1}(a_{F_j}(X)-1)\ord_{F_i}F_j^i=a_{F_i}(X),
\end{align*}
the required inequality follows from Lemma \ref{lem:algorithm}(\ref{itm:easybound}).
\end{proof}

\begin{lemma}\label{lem:c}
Fix a positive rational number $q$. Then there exists a positive integer $c$ depending only on $q$ such that every $F_i$ in Algorithm \textup{\ref{alg:canonical}} satisfies the inequality $a_{F_i}(X)\le c$.
\end{lemma}

\begin{proof}
Take a positive integer $n$ such that $nq$ is integral, and take $\epsilon$ in Lemma \ref{lem:e} and $r$ in Lemma \ref{lem:r}. Fix a positive integer $c_0$ at least $q\epsilon^{-1}$ and define positive integers $c_1,\ldots,c_r$ inductively by the recurrence relation
\begin{align*}
c_{j+1}=2+(c_j-1)n
\end{align*}
for $0\le j<r$. We shall prove that $c_r$ is a required constant.

Let $e$ be the greatest integer such that $q_e$ is defined at a smooth point $P_e$ of $X_e$. Then by Lemmata \ref{lem:algorithm}(\ref{itm:nondecr}), (\ref{itm:easybound}) and \ref{lem:e}, the estimate $a_{F_i}(X)\le c_0$ holds for any $i\le e$, and by Lemma \ref{lem:r}, if the algorithm defines $P_{e+1}\in X_{e+1}$, then $r_{e+1}\le r$. In particular, the algorithm terminates with the output $X_{e+r'}$ for some $r'\le r$ by Theorem \ref{thm:divisorial}(\ref{itm:kawamata}). Thus it is enough to show that $a_{F_{e+j}}(X)\le c_j$ for any $j\le r$ as far as $F_{e+j}$ is defined. This is reduced to proving that if $F_i$ is defined at a singular point $P_i$ of $X_i$ and if $a_{F_j}(X)$ is bounded from above by a positive integer $c'$ for all $j<i$, then $a_{F_i}(X)$ is at most $2+(c'-1)n$.

Suppose that $F_i$ and $c'$ are given as above. By Lemma \ref{lem:algorithm}(\ref{itm:lessthan1}), the $\bQ$-divisor $\Delta_i$ satisfies that $S_i\le n\Delta_i$. Since $(X_i,\Delta_i,\fa_i^q)$ is crepant to $(X,\fa^q)$, one has that
\begin{align*}
\ord_{F_i}S_i\le n\ord_{F_i}\Delta_i&=n(a_{F_i}(X_i,\fa_i^q)-a_{F_i}(X_i,\Delta_i,\fa_i^q))\\
&\le n(a_{F_i}(X_i,\fa_i^{q_i})-a_{F_i}(X,\fa^q))=n(1-a_{F_i}(X,\fa^q))\le n,
\end{align*}
where the second inequality follows from $q_i\le q$. Together with $a_{F_i}(X_i)=1+1/r_i$ by Theorem \ref{thm:divisorial}(\ref{itm:kawamata}), one computes that
\begin{align*}
a_{F_i}(X)&=a_{F_i}(X_i)+\sum_{j=0}^{i-1}(a_{F_j}(X)-1)\ord_{F_i}F_j^i\\
&\le1+\frac{1}{r_i}+(c'-1)\ord_{F_i}S_i<2+(c'-1)n.
\end{align*}
\end{proof}

In order to control the log discrepancy of a divisor computing $\mld_P(X,\fa^q)$, we need an extra assumption that $\mld_P(X,\fa^q)$ is positive.

\begin{lemma}\label{lem:process37}
Fix a positive rational number $q$. Then there exists a positive integer $l$ depending only on $q$ such that in Algorithm \textup{\ref{alg:canonical}} if $\mld_P(X,\fa^q)$ is positive and if the algorithm terminates at the process \textup{\ref{prc:notpoint}} or \textup{\ref{prc:singout}}, then there exists a divisor $E'$ over $X$ which computes $\mld_P(X,\fa)$ and satisfies the inequality $a_{E'}(X)\le l$.
\end{lemma}

\begin{proof}
\medskip
\textit{Step} 1.
We take $c$ in Lemma \ref{lem:c}. Let $\eta$ be a positive rational number such that the exist no integers $a$ satisfying that $q<1/a<q+\eta$. Let $n$ be a positive integer such that $nq$ is integral. Since Conjecture \ref{cnj:equiv}(\ref{itm:nakamura}) holds for in dimension two by Theorem \ref{thm:surface}, there exists a positive integer $l'$ depending only on $n$ such that if $Q\in H$ is the germ of a smooth surface and $\fa_H$ is an ideal on $H$, then there exists a divisor $E_H$ over $H$ which computes $\mld_Q(H,\fa_H^{1/n})$ and satisfies the inequality $a_{E_H}(H)\le l'$.

Let $\fm$ be the maximal ideal in $\sO_X$ defining $P$. Let $E'$ be an arbitrary divisor over $X$ which computes $\mld_P(X,\fa^q)$. By Corollary \ref{crl:mult}, there exists a positive integer $b$ depending only on $q$ such that $\ord_{E'}\fm\le b$. Note that $b$ can be taken independent of the germ $P\in X$ of a smooth threefold. Indeed, $E'$ also computes $\mld_P(X,\fa'^q)$ for the $\fm$-primary ideal $\fa'=\fa+\fm^e$ as far as a positive integer $e$ satisfies that $\ord_{E'}\fa\le e\ord_{E'}\fm$. Thus by Lemma \ref{lem:regular}, one can take the $b$ on the germ at origin of the affine space $\bA^3$.

For any $i$, one has the estimate $\ord_{E'}S_i\le\ord_{E'}\fm$ because $\fm^r\sO_{X_i}$ is contained in $\sO_{X_i}(-rS_i)$ for a positive integer $r$ such that $rS_i$ is Cartier. Hence,
\begin{align*}
\ord_{E'}S_i\le b.
\end{align*}

Supposing that the algorithm terminates at the process \ref{prc:notpoint} or \ref{prc:singout}, we shall bound the log discrepancy of some divisor which computes $\mld_P(X,\fa)$ in terms of $q$, $c$, $\eta$, $l'$ and $b$.

\medskip
\textit{Step} 2.
Suppose that the algorithm terminates at the process \ref{prc:notpoint} and outputs $X_i$. Then the centre $c_E(X_i)$ on $X_i$ of $E$ is either a divisor or a curve. If it is a divisor, then $E=F_{i-1}$ and it computes $\mld_P(X,\fa^q)$. By Lemma \ref{lem:c}, $F_{i-1}$ satisfies that
\begin{align*}
a_{F_{i-1}}(X)\le c.
\end{align*}

Suppose that $c_E(X_i)$ is a curve $C$. Let $H$ be a general hyperplane section of $X_i$ and $Q$ be a closed point in $H\cap C$. Considering a log resolution, one has that
\begin{align*}
\mld_Q(H,\Delta_i|_H,(\fa_i\sO_H)^q)=\mld_{\eta_C}(X_i,\Delta_i,\fa_i^q)=\mld_P(X,\fa^q),
\end{align*}
where the second equality holds since $E$ computes $\mld_P(X,\fa^q)$. Moreover by the expression $\mld_Q(H,\Delta_i|_H,(\fa_i\sO_H)^q)=\mld_Q(H,(\fa_i^{nq}\sO_H(-n\Delta_i|_H))^{1/n})$, there exists a divisor $E'$ over $X_i$ with $c_{X_i}(E')=C$ such that an irreducible component $E'_H$ of $E'\times_{X_i}H$ mapped to $Q$ computes $\mld_Q(H,\Delta_i|_H,(\fa_i\sO_H)^q)$ and satisfies the inequality $a_{E'_H}(H)\le l'$. The $E'$ computes $\mld_P(X,\fa^q)$ as well as $\mld_{\eta_C}(X_i,\Delta_i,\fa_i^q)$, and $a_{E'}(X_i)=a_{E'_H}(H)\le l'$. Therefore,
\begin{align*}
a_{E'}(X)=a_{E'}(X_i)+\sum_{j=0}^{i-1}(a_{F_j}(X)-1)\ord_{E'}F_j^i\le l'+(c-1)\ord_{E'}S_i\le l'+(c-1)b,
\end{align*}
where the last inequality follows from Step 1.

\medskip
\textit{Step} 3.
Suppose that the algorithm terminates at the process \ref{prc:singout} and outputs $X_i$. Let $\fn$ be the maximal ideal in $\sO_{X_i}$ defining $P_i$. Recall that $\fa_i=\fb_i^{1/r_i}$. Set $\fb'_i=\fb_i+\fn^e$ for a positive integer $e$ such that $\ord_E\fb_i\le e\ord_E\fn$. Since $\mld_{P_i}(X_i,\fa_i^q)$ is greater than one, there exists a rational number $q'$ greater than $q$ such that $\mld_{P_i}(X_i,(\fb'_i)^{q'/r_i})=1$. There exists a divisorial contraction to $X_i$ crepant to $(P_i,X_i,(\fb'_i)^{q'/r_i})$ by Lemma \ref{lem:crepant} and it is uniquely determined by Theorem \ref{thm:divisorial}(\ref{itm:kawamata}). Its exceptional divisor $F$ satisfies that $a_F(X_i)=1+1/r_i$. Thus
\begin{align*}
q'\ord_F\fb'_i=r_i(a_F(X_i)-a_F(X_i,(\fb'_i)^{q'/r_i}))=r_i(a_F(X_i)-1)=1,
\end{align*}
which derives that $q'$ is at least $q+\eta$ by the definition of $\eta$. In particular, the pair $(X_i,(\fb'_i)^{(q+\eta)/r_i})$ is canonical so $a_E(X_i,\fa_i^{q+\eta})=a_E(X_i,(\fb'_i)^{(q+\eta)/r_i})\ge1$. Hence one computes that
\begin{align*}
a_E(X_i)&=a_E(X_i,\fa_i^q)+q\ord_E\fa_i\\
&=a_E(X_i,\fa_i^q)+q\eta^{-1}(a_E(X_i,\fa_i^q)-a_E(X_i,\fa_i^{q+\eta}))\\
&\le(1+q\eta^{-1})a_E(X_i,\fa_i^q)-q\eta^{-1}\\
&=(1+q\eta^{-1})(a_E(X_i,\Delta_i,\fa_i^q)+\ord_E\Delta_i)-q\eta^{-1}\\
&\le(1+q\eta^{-1})(a_E(X,\fa^q)+\ord_ES_i)-q\eta^{-1}
\end{align*}
and
\begin{align*}
a_E(X)&=a_E(X_i)+\sum_{j=0}^{i-1}(a_{F_j}(X)-1)\ord_EF_j^i\\
&\le(1+q\eta^{-1})(a_E(X,\fa^q)+\ord_ES_i)-q\eta^{-1}+(c-1)\ord_ES_i.
\end{align*}
Together with $a_E(X,\fa^q)=\mld_P(X,\fa^q)\le3$ and $\ord_ES_i\le b$ in Step 1, one concludes that
\begin{align*}
a_E(X)\le(1+q\eta^{-1})(3+b)-q\eta^{-1}+(c-1)b.
\end{align*}

\medskip
\textit{Step} 4.
By Steps 2 and 3, any integer $l$ at least $c$, $l'+(c-1)b$ and $(1+q\eta^{-1})(3+b)-q\eta^{-1}+(c-1)b$ satisfies the required property.
\end{proof}

\begin{proof}[Proof of Theorem \textup{\ref{thm:canonical}}]
We shall verify that the $l$ in Lemma \ref{lem:process37} and $c$ in Lemma \ref{lem:c} satisfy the assertion. Let $\fa$ be an ideal on $X$ such that $\mld_P(X,\fa^q)$ is positive. Run Algorithm \ref{alg:canonical} which terminates by Proposition \ref{prp:termination}. If the algorithm terminates at the process \ref{prc:notpoint} or \ref{prc:singout}, then the property (\ref{itm:bounded}) holds by Lemma \ref{lem:process37}. If it terminates at the process \ref{prc:smoothout}, then let $Q\in Y$ be the output $P_i\in X_i$. The $Q\in Y$ satisfies the property (\ref{itm:reduced}) by Lemmata \ref{lem:algorithm}(\ref{itm:lessthan1}) and \ref{lem:c}.
\end{proof}

\section{Extraction by weighted blow-ups}
Recall the classification of divisors over a smooth surface computing the minimal log discrepancy.

\begin{theorem}[\cite{K17}]\label{thm:mldwbu}
Let $P\in X$ be the germ of a smooth surface and $\fa$ be an $\bR$-ideal on $X$.
\begin{enumerate}
\item
If $(X,\fa)$ is lc, then every divisor computing $\mld_P(X,\fa)$ is obtained by a weighted blow-up.
\item
If $(X,\fa)$ is not lc, then some divisor computing $\mld_P(X,\fa)$ is obtained by a weighted blow-up.
\end{enumerate}
\end{theorem}

We want to apply this theorem with the object of extracting by a weighted blow-up a divisor over a smooth threefold whose centre is a curve and which computes the lc threshold. In order to use such extraction in the study of the generic limit of ideals, we need to formulate it for $R$-varieties. We let $K$ be a field of characteristic zero throughout this section. The purpose of this section is to prove

\begin{theorem}\label{thm:wbu}
Let $X$ be the spectrum of the ring of formal power series in three variables over $K$ and $P$ be its closed point. Let $\fa$ be an $\bR$-ideal on $X$ such that
\begin{itemize}
\item
$\mld_P(X,\fa)$ equals one, and
\item
$(X,\fa)$ has an lc centre $C$ of dimension one.
\end{itemize}
Then there exist a divisor $E$ over $X$ computing $\mld_{\eta_C}(X,\fa)$ and a part $x_1,x_2$ of a regular system of parameters in $\sO_X$ such that $E$ is obtained by the weighted blow-up of $X$ with $\wt(x_1,x_2)=(w_1,w_2)$ for some coprime positive integers $w_1,w_2$.
\end{theorem}

When a divisor over a smooth variety is given, we often realise it by a finite sequence of blow-ups.

\begin{definition}\label{dfn:tower}
Let $X$ be a smooth variety and $E$ be a divisor over $X$ whose centre $Z$ on $X$ has codimension at least two in $X$. A \textit{tower} on $X$ with respect to $E$ is a finite sequence of projective birational morphisms $X_{i+1}\to X_i$ of smooth varieties for $0\le i<l$ such that
\begin{itemize}
\item
$X_0=X$ and $Z_0=Z$,
\item
$X_{i+1}$ is about $\eta_{Z_i}$ the blow-up of $X_i$ along $Z_i$,
\item
$E_i$ is the exceptional prime divisor on $X_{i+1}$ contracting onto $Z_i$,
\item
$Z_{i+1}$ is the centre on $X_{i+1}$ of $E$, and
\item
$E_{l-1}=E$.
\end{itemize}
A tower is called the \textit{regular tower} if for any $i<l$, the centre $Z_i$ is smooth and $X_{i+1}$ is globally the blow-up of $X_i$ along $Z_i$. Note that the regular tower is uniquely determined by $E$ if it exists.
\end{definition}

\begin{remark}\label{rmk:toric}
Let $P\in X$ be the germ of a smooth variety. Let $x_1,\ldots,x_c$ be a part of a regular system of parameters in $\sO_X$ and $E$ be the divisor obtained by the weighted blow-up of $X$ with $\wt(x_1,\ldots,x_c)=(w_1,\ldots,w_c)$, where $c$ is at least two. Then one can see that the regular tower on $X$ with respect to $E$ exists in terms of toric geometry. Following the notation in \cite{I14}, set $N=\bZ^d$ with the standard basis $e_1,\ldots,e_d$ for $d=\dim X$. One may assume that $w=(w_1,\ldots,w_c,0,\ldots,0)$ is primitive in $N$. Construct a finite sequence of fans $(N,\Delta_i)$ for $0\le i\le l$ such that
\begin{itemize}
\item
$I_i=\{e_1,\ldots,e_d\}\cup\{v_1,\ldots,v_i\}$,
\item
$\Delta_i$ is the set of all cones spanned by a subset of $I_i$,
\item
$J_i$ is the smallest subset of $I_i$ such that $w$ belongs to the cone spanned by $J_i$,
\item
$v_{i+1}=\sum_{v\in J_i}v$, and
\item
$J_i\neq\{w\}$ for $i<l$ and $J_l=\{w\}$.
\end{itemize}
Set the toric variety $T_i=T_N(\Delta_i)$ and let $E_i^T$ be the exceptional divisor of $T_{i+1}\to T_i$. Then $X$ has an \'etale morphism to $T_0$ by corresponding $e_i$ to $x_i$. The base changes of $T_i$ to $X$ form the regular tower on $X$ with respect to $E$, and every $E_i=E_i^T\times_{T_0}X$ is obtained by a weighted blow-up of $X$.
\end{remark}

We collect basic properties of the log discrepancies in a tower which was essentially written in \cite[Proposition 6]{K17}.

\begin{lemma}\label{lem:tower}
Notation as in Definition \textup{\ref{dfn:tower}}. Let $\fa$ be an $\bR$-ideal on $X$ and $\fa_i$ be its weak transform on $X_i$. Set $a_i=a_{E_i}(X,\fa)$.
\begin{enumerate}
\item\label{itm:order}
The $\ord_{Z_i}\fa_i$ form a non-increasing sequence.
\item\label{itm:order1}
If $\ord_Z\fa\le1$, then $a_i\ge1$ and the $a_i$ form a non-decreasing sequence.
\item\label{itm:ordlthan1}
If $\ord_Z\fa<1$, then $a_i>1$ and the $a_i$ form a strictly increasing sequence.
\end{enumerate}
\end{lemma}

\begin{proof}
Take a subvariety $V_{i+1}$ of $Z_{i+1}$ such that $V_{i+1}\to Z_i$ is finite and dominant. Then $\ord_{Z_{i+1}}\fa_{i+1}\le\ord_{V_{i+1}}\fa_{i+1}\le\ord_{Z_i}\fa_i$ by \cite[III Lemmata 7 and 8]{H64}, which is (\ref{itm:order}).

The assertion (\ref{itm:order1}) is reduced to (\ref{itm:ordlthan1}) because $a_i$ is the limit of $a_{E_i}(X,\fa^{1-\epsilon})$ when $\epsilon$ decreases to zero. Suppose that $\ord_Z\fa<1$ in order to see (\ref{itm:ordlthan1}). Then so are $\ord_{Z_i}\fa_i$ by (\ref{itm:order}). In particular, $a_{E_i}(X_i,\fa_i)=a_{E_i}(X_i)-\ord_{Z_i}\fa_i>1$. Since $(X_i,\sum_{j=0}^{i-1}(1-a_j)E_j^i,\fa_i)$ is crepant to $(X,\fa)$, where $E_j^i$ is the strict transform of $E_j$, one computes that
\begin{align*}
a_i=a_{E_i}(X_i,\fa_i)+\sum_{j=0}^{i-1}(a_j-1)\ord_{E_i}E_j^i>1+\sum_{j=0}^{i-1}(a_j-1)\ord_{E_i}E_j^i.
\end{align*}
This derives that $a_i>1$ by induction, and then derives that $a_i>a_{i-1}$ again by induction.
\end{proof}

\begin{proposition}\label{prp:tower}
Let $P\in X$ be the germ of a smooth variety and $\fa$ be an $\bR$-ideal on $X$. Let $E$ be the divisor obtained by the blow-up of $X$ at $P$.
\begin{enumerate}
\item\label{itm:Ecomp}
If $\ord_P\fa\le1$, then $E$ computes $\mld_P(X,\fa)$.
\item
If $\ord_P\fa<1$, then $E$ is the unique divisor computing $\mld_P(X,\fa)$.
\end{enumerate}
\end{proposition}

\begin{proof}
It is \cite[Proposition 6]{K17} exactly. Just apply Lemma \ref{lem:tower}(\ref{itm:order1}) and (\ref{itm:ordlthan1}) to the tower on $X$ with respect to a divisor which computes $\mld_P(X,\fa)$.
\end{proof}

We shall study divisors computing the minimal log discrepancy on a $K$-variety of dimension two.

\begin{lemma}\label{lem:Kpoint}
Let $P\in X$ be the germ at a $K$-point of a regular $K$-variety of dimension two and $\fa$ be an $\bR$-ideal on $X$. Then there exists a divisor over $X$ computing $\mld_P(X,\fa)$ which is obtained by a sequence of finitely many blow-ups at a $K$-point.
\end{lemma}

\begin{proof}
Let $\fm$ be the maximal ideal in $\sO_X$ defining $P$. By adding a high multiple of $\fm$ to each component of $\fa$, we may assume that $\fa$ is $\fm$-primary. Suppose that $(X,\fa)$ is not lc. Then there exists a positive real number $t$ less than one such that $\mld_P(X,\fa^t)$ is zero. Replacing $\fa$ with $\fa^t$, we may assume that $(X,\fa)$ is lc.

We write $m=\mld_P(X,\fa)$ for simplicity. Let $Y$ be the blow-up of $X$ at $P$ and $E$ be its exceptional divisor. There is nothing to prove if $E$ computes $\mld_P(X,\fa)$. Thus we may assume that $a_E=a_E(X,\fa)$ is greater than $m$. Then $a_E=2-\ord_P\fa<1$ by Proposition \ref{prp:tower}(\ref{itm:Ecomp}) (which also holds for regular $K$-varieties by Remark \ref{rmk:regular}). That is, $m< a_E<1$. Let $\fa_Y$ be the weak transform on $Y$ of $\fa$, then $(Y,(1-a_E)E,\fa_Y)$ is crepant to $(X,\fa)$. We claim that there exists a unique point $Q$ in $Y$ such that $\mld_Q(Y,(1-a_E)E,\fa_Y)=m$, and claim that $Q$ is a $K$-point.

Let $Q$ be an arbitrary closed point in $Y$ such that $\mld_Q(Y,(1-a_E)E,\fa_Y)=m$. Such $Q$ exists since $a_E\neq m$. Set the base change $\bar X=X\times_{\Spec K}\Spec\bar K$ of $X$ to the algebraic closure $\bar K$ of $K$. Let $\bar P$, $\bar\fa$, $\bar Y$, $\bar E$ and $\bar\fa_Y$ be the base changes of $P$, $\fa$, $Y$, $E$ and $\fa_Y$ to $\bar K$ as well. Then every closed point $\bar Q$ in $Q\times_X\bar X$ satisfies that $\mld_{\bar Q}(\bar Y,(1-a_E)\bar E,\bar\fa_Y)=\mld_{\bar P}(\bar X,\bar\fa)$. Thus our claims on $Q$ come from those on $\bar Q$, so we may assume that $K$ is algebraically closed.

One has that $\mld_Q(Y,E,\fa_Y)\le\mld_Q(Y,(1-a_E)E,\fa_Y)-a_E=m-a_E<0$, which means that $(Y,E,\fa_Y)$ is not lc at $Q$. By inversion of adjunction, $(E,\fa_Y\sO_E)$ is not lc at $Q$, that is, $\ord_Q(\fa_Y\sO_E)>1$. Hence the number of $Q$ is less than the degree of the divisor on $E\simeq\bP_K^1$ defined by $\fa_Y\sO_E$, which equals $\ord_E\fa=2-a_E$. Thus, the uniqueness of $Q$ follows.

While $a_E$ is greater than $m$, we replace $P\in(X,\fa)$ with $Q\in(Y,\sO_Y(-E)^{1-a_E}\cdot\fa_Y)$ and repeat the same argument. This procedure terminates at finitely many times. Indeed, let $l$ be the minimum of $a_F(X)$ for all divisors $F$ over $X$ computing $\mld_P(X,\fa)$. Then after at most $(l-1)$ blow-ups, one attains a divisor which computes $\mld_P(X,\fa)$.
\end{proof}

\begin{example}
There may exist a divisor computing $\mld_P(X,\fa)$ which is not obtained by a sequence of blow-ups at a $K$-point. For example, let $P\in\bA_\bR^2$ be the germ at origin of the affine plane over $\bR$ with coordinates $x_1$, $x_2$, and $H$ be the divisor on $\bA_\bR^2$ defined by $x_1^2+x_2^2$. Then $\mld_P(\bA_\bR^2,H)=0$. Let $Y$ be the blow-up of $\bA_\bR^2$ at $P$ and $E$ be its exceptional divisor. Then $(Y,H_Y+E)$ is crepant to $(\bA_\bR^2,H)$, where $H_Y$ is the strict transform. The intersection $Q$ of $H_Y$ and $E$ is a $\bC$-point such that $\mld_Q(Y,H_Y+E)=0$.
\end{example}

Now we apply Theorem \ref{thm:mldwbu} to $K$-varieties of dimension two.

\begin{proposition}\label{prp:wbu}
Let $P\in X$ be the germ at a $K$-point of a regular $K$-variety of dimension two and $\fa$ be an $\bR$-ideal on $X$. Then there exists a divisor $E$ over $X$ computing $\mld_P(X,\fa)$ which is obtained by a weighted blow-up.
\end{proposition}

\begin{proof}
We may assume the log canonicity of $(X,\fa)$ by the argument in the first paragraph of the proof of Lemma \ref{lem:Kpoint}. By Lemma \ref{lem:Kpoint}, there exists a divisor $E$ over $X$ computing $\mld_P(X,\fa)$ which is obtained by a sequence of finitely many blow-ups at a $K$-point. Set the base change $\bar X=X\times_{\Spec K}\Spec\bar K$ of $X$ to the algebraic closure $\bar K$ of $K$ and let $\bar P$ and $\bar\fa$ be the base changes of $P$ and $\fa$ to $\bar X$. Since $E$ is obtained by finitely many blow-ups at a $K$-point, its base change $\bar E=E\times_X\bar X$ is irreducible, so $\bar E$ is a divisor over $\bar X$. Thus by Theorem \ref{thm:mldwbu}, there exists a regular system $x_1,x_2$ of parameters in $\sO_{\bar X,\bar P}$ such that $\bar E$ is obtained by the weighted blow-up of $\bar X$ with $\wt(x_1,x_2)=(w_1,w_2)$ for some coprime positive integers $w_1,w_2$.

We shall show that one can take $x_1$ and $x_2$ from $\sO_X$. This is obvious when $w_1=w_2=1$ because the weighted blow-up in this case is nothing but the blow-up at the point. Suppose that $w_1>w_2$. Let $L$ be a finite Galois extension of $K$ such that $x_1$ and $x_2$ belong to $\sO_X\otimes_KL$. Then for any element $\sigma$ of the Galois group $G$ of $L/K$, the $\bar E=\bar E^\sigma$ is obtained by the weighted blow-up with $\wt(x_1^\sigma,x_2^\sigma)=(w_1,w_2)$. Thus one can replace $x_i$ with its trace $\sum_{\sigma\in G}x_i^\sigma$ by Remark \ref{rmk:wbu}. Here one can assume that $\sum_{\sigma\in G}x_i^\sigma\in\fm\setminus\fm^2$, where $\fm$ is the maximal ideal in $\sO_X$ defining $P$, by replacing $x_i$ with $\lambda_ix_i$ for a general member $\lambda_i$ in $L$.

Now we may assume that $x_1$ and $x_2$ belong to $\sO_X$. Then $E$ is obtained by the weighted blow-up of $X$ with $\wt(x_1,x_2)=(w_1,w_2)$.
\end{proof}

\begin{proof}[Proof of Theorem \textup{\ref{thm:wbu}}]
\textit{Step} 1.
First of all, remark that $C$ is the smallest lc centre of $(X,\fa)$. The $C$ is regular by \cite[Theorem 1.2]{K15}, and it is geometrically irreducible because its base change to any field is again the smallest lc centre of the base change of $(X,\fa)$. Thus, there exists a regular system $x_1,x_2,x_3$ of parameters in $\sO_X$ such that the ideal $\sI_C$ in $\sO_X$ defining $C$ is generated by $x_1$ and $x_2$. If we consider instead of $\fa=\prod_j\fa_j^{r_j}$ the $\bR$-ideal $\fb=\prod_j(\fa_j+(x_1,x_2)^l\sO_X)^{r_j}$ for a large integer $l$, then $C$ is still the smallest lc centre of $(X,\fb)$ and $\mld_P(X,\fb)\ge\mld_P(X,\fa)=1$. On the other hand, $\mld_P(X,\fb)$ is at most one by \cite[Proposition 6.1]{K15}. Thus $\mld_P(X,\fb)$ must equal one.

Hence by replacing $\fa$ with $\fb$, we may assume that $\fa$ is the pull-back of an $\bR$-ideal $\fa'$ on $X'=\Spec K[[x_3]][x_1,x_2]$. Set $X''=\Spec K((x_3))[x_1,x_2]$, where $K((x_3))$ is the quotient field of $K[[x_3]]$. There exist natural morphisms
\begin{align*}
X\to X'\leftarrow X''.
\end{align*}
Let $P'$ be the point of $X'$ defined by $(x_1,x_2,x_3)\sO_{X'}$ and $P''$ be the point of $X''$ defined by $(x_1,x_2)\sO_{X''}$.

One has that $\mld_{P''}(X'',\fa'\sO_{X''})=\mld_{\eta_C}(X,\fa)=0$. By Proposition \ref{prp:wbu}, there exists a divisor $E''$ over $X''$ computing $\mld_{P''}(X'',\fa'\sO_{X''})$ which is obtained by a weighted blow-up of $X''$. Let $E'$ be the unique divisor over $X'$ such that $E''=E'\times_{X'}X''$ and let $E=E'\times_{X'}X$. Note that $C$ is the centre of $E$ on $X$.

\medskip
\textit{Step} 2.
There exists a regular tower $\cT''$ on $X$ with respect to $E''$ in Definition \ref{dfn:tower} (which can be extended to $K((x_3))$-varieties). As seen in Remark \ref{rmk:toric}, $\cT''$ is a finite sequence $X''_l\to\cdots\to X''_0=X''$ of blow-ups at a $K((x_3))$-point and the exceptional divisor $F''_i$ of $X''_{i+1}\to X''_i$ is obtained by a weighted blow-up of $X''$. Note that $E''=F''_{l-1}$. Possibly by replacing $E''$ with some $F''_i$, we may assume that $F_i''$ does not compute $\mld_{P''}(X'',\fa'\sO_{X''})$ unless $i=l-1$.

The $\cT''$ is compactified over $X'$, that is, $X''_{i+1}\to X''_i$ is the base change of a projective birational morphism $X'_{i+1}\to X'_i$ of regular schemes. Then the base changes $X_i=X'_i\times_{X'}X$ to $X$ form a tower $\cT$ on $X$ with respect to $E$. Let $C_i$ be the centre on $X_i$ of $E$. Since $\cT''$ consists of blow-ups at a $K$-point, $C_i$ is birational to $C$ for any $i<l$. Hence $C_i$ must be isomorphic to the regular scheme $C$. Therefore one can replace $X_i$ and $X'_i$ inductively so that $\cT$ is the regular tower on $X$ with respect to $E$.

Let $F_i$ denote the exceptional divisor of $X_{i+1}\to X_i$, and set $a_i=a_{F_i}(X,\fa)$. By our construction, every $a_i$ is positive except for $i=l-1$ while $a_{l-1}$ is zero. Let $\fa_i$ be the weak transform on $X_i$ of $\fa$ and set the $\bR$-divisor $\Delta_i=\sum_{j=0}^{i-1}(1-a_j)F_j^i$ on $X_i$, where $F_j^i$ is the strict transform of $F_j$. Then $(X_i,\Delta_i,\fa_i)$ is crepant to $(X,\fa)$. We claim that $a_i<1$ for any $i$. This is obvious for $i=l-1$ since $a_{l-1}=0$. In order to see the inequality $a_i<1$ for the fixed index $i<l-1$ by induction, assume that $a_j<1$ for any $j$ less than $i$. Then $\Delta_i$ is effective. Since $F_i$ does not compute $\mld_{\eta_C}(X,\fa)$, one has that $\ord_{F_i}\Delta_i+\ord_{F_i}\fa_i>1$ by Proposition \ref{prp:tower}. Hence one obtains that $a_i=a_{F_i}(X_i,\Delta_i,\fa_i)=2-(\ord_{F_i}\Delta_i+\ord_{F_i}\fa_i)<1$.

\medskip
\textit{Step} 3.
We have that $0<a_i<1$ for any $i<l-1$ while $a_{l-1}=0$. We let $f_i$ denote the fibre of $F_i\to C$ at $P$, which is isomorphic to $\bP_K^1$. For $i<l$, let $P_i$ be the $K$-point in $C_i$ mapped to $P$. We claim that for any indices $i$ and $j$ such that $j<i<l$, the centre $C_i$ is either disjoint from $F_j^i$ or contained in $F_j^i$.

Indeed if $C_i$ intersected $F_j^i$ properly at $P_i$, then the morphism $F_j^{i+1}\to F_j^i$ would not be an isomorphism. Thus $F_j^{i+1}$ must contain the fibre $f_i$ of $F_i\to C_i$. In particular, $F_j^{i+1}$ intersects $C_{i+1}$. On the other hand, $C_{i+1}$ is not contained in $F_j^{i+1}$ as $C_i$ is not in $F_j^i$. Thus one obtains that $C_{i+1}$ must also intersect $F_j^{i+1}$ properly at $P_{i+1}$, unless $i+1=l$. Repeating this argument, one would have that $F_j^l$ contains $f_{l-1}$ as well as $F_{l-1}$ contains $f_{l-1}$. Now let $G$ be the divisor obtained by the blow-up of $X_l$ along $f_{l-1}$. One computes that
\begin{align*}
a_G(X,\fa)\le a_G(X_l,\Delta_l)\le2-(1-a_j)-(1-a_{l-1})=a_j<1,
\end{align*}
which contradicts that $\mld_P(X,\fa)=1$.

\medskip
\textit{Step} 4.
Let $i$ be any index such that $\ord_{F_i}\sI_C=1$. We shall show that there exists a part $y_1$ of a regular system of parameters in $\sO_X$ such that $C_i$ is contained in the strict transform $H_i$ on $X_i$ of the divisor on $X$ defined by $y_1$. This is obvious for $i=0$. The condition $\ord_{F_i}\sI_C=1$ for the fixed $i\ge1$ implies that $\ord_{F_{i-1}}\sI_C=1$ since
\begin{align*}
\ord_{F_i}\sI_C=\ord_{F_i}\sI_i+\sum_{j=0}^{i-1}\ord_{F_j}\sI_C\cdot\ord_{F_i}{F_j^i}\ge\ord_{F_{i-1}}\sI_C
\end{align*}
for the weak transform $\sI_i$ on $X_i$ of $\sI_C$. Hence by induction on $i$, we may assume the existence of $y_1$ such that $C_{i-1}$ is contained in $H_{i-1}$.

We extend $y_1$ to a regular system $y_1,y_2,x_3$ of parameters in $\sO_X$ in which $y_2$ is a general member in $\sI_C$. Then for any $j\le i-1$, $F_j$ is as a divisor over $X$ obtained by the weighted blow-up of $X$ with $\wt(y_1,y_2)=(j+1,1)$, and the $y_1/y_2^j,y_2,x_3$ form a regular system of parameters in $\sO_{X_j,P_j}$. In particular, $f_{i-1}\simeq\bP_K^1$ has homogeneous coordinates $y_1/y_2^{i-1}$, $y_2$. Moreover, the $K$-point $P_i\in f_{i-1}$ is not defined by $[y_1/y_2^{i-1}:y_2]=[1:0]$. This follows when $i=1$ from the general choice of $y_2$, and when $i\ge2$ from the property in Step 3 that $C_i$ does not intersect $F_{i-2}^i$. Take $c\in K$ such that $P_i\in f_{i-1}$ is defined by $[y_1/y_2^{i-1}:y_2]=[c:1]$. Replacing $y_1$ with $y_1-cy_2^i$, we may assume that $c=0$.

Then $y_1/y_2^i,y_2,x_3$ form a regular system of parameters in $\sO_{X_i,P_i}$. The $H_i$, $F_{i-1}$ and $f_{i-1}$ are defined at $P_i$ by $y_1/y_2^i$, $y_2$ and $(y_2,x_3)\sO_{X_i}$. Because the fibration $F_{i-1}\to C_{i-1}$ is isomorphic to the projection of the product $\bP_K^1\times_{\Spec K}C_{i-1}$, its section $C_i$ is defined at $P_i$ by $(y_1/y_2^i+x_3v(x_3),y_2)\sO_{X_i}$ for some $v(x_3)\in K[[x_3]]$. After replacing $y_1$ with $y_1+y_2^ix_3v(x_3)$, one has that $C_i$ is contained in $H_i$.

\medskip
\textit{Step} 5.
Let $e$ be the maximal index such that $\ord_{F_e}\sI_C=1$ and choose a regular system $y_1,y_2,x_3$ of parameters in $\sO_X$ such that $y_1$ satisfies the condition in Step 4 for $i=e$ and $y_2$ is a general member in $\sI_C$. Now repeating the process in Step 1 for $y_1,y_2,x_3$ instead of $x_1,x_2,x_3$, we may assume that $x_1=y_1$ and $x_2=y_2$. Then by Remark \ref{rmk:toric}, one can obtain $E''=F''_{l-1}$ by a weighted blow-up with respect to the coordinates $x_1,x_2$. More precisely, there exist a non-negative integer $p$ and a positive integer $q$ such that $E''$ is obtained by the weighted blow-up of $X''$ with $\wt(x_1,x_2)=p(e,1)+q(e+1,1)$. Note that $p$ is positive iff $e+1<l$.

Therefore, we conclude that $E$ is also obtained by the weighted blow-up of $X$ with $\wt(x_1,x_2)=p(e,1)+q(e+1,1)$.
\end{proof}

\section{Reduction to the case of decomposed boundaries}\label{sct:reduction}

The objective of this section is to complete the reduction to Conjecture \ref{cnj:product}.

\begin{remark}\label{rmk:independent}
In order to prove Conjecture \ref{cnj:product} or a statement of the same kind, it is sufficient to find an integer $l$ which satisfies the required property but may depend on the germ $P\in X$ of a smooth threefold, for the reason that one has only to consider those $\fa$ which are $\fm$-primary. Indeed, there exists an \'etale morphism from $P\in X$ to the germ $o\in\bA^3$ at origin of the affine space. Then as it is seen in Lemma \ref{lem:regular}, any $\fm$-primary ideal $\fa$ on $X$ is the pull-back of some ideal $\fb$ on $\bA^3$, and $\mld_P(X,\fa^q\fm^s)$ coincides with $\mld_o(\bA^3,\fb^q\fn^s)$, where $\fn$ is the maximal ideal in $\sO_{\bA^3}$ defining $o$. Thus, the bound $l$ on the germ $o\in\bA^3$ can be applied to an arbitrary germ $P\in X$.
\end{remark}

We shall make the reduction by using the generic limit of ideals. For a moment, we work in the general setting that $P\in X$ is the germ of a klt variety. Let $r_1,\ldots,r_e$ be positive real numbers and $\cS=\{\fa_i=\prod_{j=1}^e\fa_{ij}^{r_j}\}_{i\in\bN}$ be a sequence of $\bR$-ideals on $X$. Let $\sfa=\prod_{j=1}^e\sfa_j^{r_j}$ be a generic limit of $\cS$. We use the notation in Section \ref{sct:limit}. The $\sfa$ is the generic limit with respect to a family $\cF=(Z_l,(\fa_j(l))_j,N_l,s_l,t_l)_{l\ge l_0}$ of approximations of $\cS$, and $\sfa$ is an ideal on $\hat P\in\hat X$ where $\hat X$ is the spectrum of the completion of the local ring $\sO_{X,P}\otimes_kK$.

Let $\hat f\colon\hat Y\to\hat X$ be a projective birational morphism isomorphic outside $\hat P$. Suppose that $\hat Y$ is klt and the exceptional locus of $\hat f$ is a $\bQ$-Cartier prime divisor $\hat F$. Let $\hat C$ be a closed proper subset of $\hat F$. As in Remark \ref{rmk:descend}, after replacing $\cF$ with a subfamily, for any $l\ge l_0$ the $\hat f$ is descended to a projective morphism $f_l\colon Y_l\to X\times Z_l$ from a klt variety whose exceptional locus is a $\bQ$-Cartier prime divisor $F_l$. One may assume that for any $i\in N_l$, the fibre $f_i\colon Y_i\to X$ at $s_l(i)\in Z_l$ is a morphism from a klt variety whose exceptional locus is a $\bQ$-Cartier prime divisor $F_i$. Refer to \cite[Section B]{dFEM11} for the properties of a family of normal $\bQ$-Gorenstein rational singularities. We may assume that $\hat C$ is descended to a closed subset $C_l$ in $F_l$. The $f_i$, $F_i$ and $C_i=C_l\times_{Y_l}Y_i$ are independent of $l$ because they are compatible with $t_l$.

\begin{lemma}\label{lem:relative}
Notation and assumptions as above. Suppose that $a_{\hat F}(\hat X,\sfa)$ is at most one and that the intersection of $\hat F$ and the non-klt locus on $\hat Y$ of $(\hat X,\sfa)$ is contained in $\hat C$. Then there exists a positive integer $l$ depending only on $\sfa$ and $\hat f$ such that after replacing $\cF$ with a subfamily, for any $i\in N_{l_0}$ if a divisor $E$ over $X$ computes $\mld_P(X,\fa_i)$ and has centre $c_{Y_i}(E)$ not contained in $C_i$, then $a_E(X)\le l$.
\end{lemma}

\begin{proof}
Let $r$ be a positive integer such that $r\hat F$ is Cartier. We may assume that $rF_l$ is Cartier. By replacing $\fa_{ij}$ with $(\fa_{ij})^r$ and $r_j$ with $r_j/r$, we may assume that $\fa_{ij}$ is an ideal to the power of $r$ and so is $\sfa_j$. Thus one can define the weak transform $\fa_{iY}=\prod_j(\fa_{ijY})^{r_j}$ on $Y_i$ of $\fa_i$, as well as the weak transform $\sfa_Y$ on $\hat Y$ of $\sfa$. We may assume that $\ord_{\hat F}\sfa_j=\ord_{F_l}\fa_j(l)=\ord_{F_i}\fa_{ij}<l$ for any $i\in N_l$ and $j$. Set $\fa_{jY}(l)=\fa_j(l)\sO_{Y_l}(a_jF_l)$ and $\fa_Y(l)=\prod_j(\fa_{jY}(l))^{r_j}$ for $a_j=\ord_{\hat F}\sfa_j$, which is divisible by $r$.

One can fix a positive real number $t$ such that the intersection of $\hat F$ and the non-klt locus on $\hat Y$ of $(\hat X,\sfa^{1+t})$ is still contained in $\hat C$. Set the real number $b$ so that $(\hat Y,b\hat F,\sfa_Y^{1+t})$ is crepant to $(\hat X,\sfa^{1+t})$, then $0\le b<1$ by $a_{\hat F}(\hat X,\sfa)\le1$. Then $(Y_l,bF_l,\fa_Y(l)^{1+t})$ is crepant to $(X\times Z_l,\fa(l)^{1+t})$ while $(Y_i,bF_i,(\fa_{iY})^{1+t})$ is crepant to $(X,\fa_i^{1+t})$. One may assume that $(Y_l,bF_l,\fa_Y(l)^{1+t})$ is klt about $F_l\setminus C_l$.

Apply Corollary \ref{crl:relative} to the family $Y_l\setminus C_l\to Z_l$. Since $\fa_{ijY}+\fm^l\sO_{Y_i}(a_jE_i)=\fa_{jY}(l)\sO_{Y_i}$, one has that $(Y_i,bF_i,(\fa_{iY})^{1+t})$ is lc about $F_i\setminus C_i$ for any $i\in N_{l_0}$ after replacing $\cF$ with a subfamily. Thus if a divisor $E$ over $X$ satisfies that $c_{Y_i}(E)\not\subset C_i$, then $a_E(X,\fa_i^{1+t})\ge0$, that is, $t\ord_E\fa_i\le a_E(X,\fa_i)$. Hence,
\begin{align*}
a_E(X)=a_E(X,\fa_i)+\ord_E\fa_i\le(1+t^{-1})a_E(X,\fa_i).
\end{align*}
In addition if $E$ computes $\mld_P(X,\fa_i)$, then
\begin{align*}
a_E(X)\le(1+t^{-1})\mld_P(X,\fa_i)\le(1+t^{-1})\mld_PX.
\end{align*}
Hence any integer $l$ at least $(1+t^{-1})\mld_PX$ satisfies the required property.
\end{proof}

We provide a meta theorem which connects statements involving the maximal ideal $\fm$ to those involving $\fm$-primary ideals. For the property $\sP$ in the theorem, one can take for example empty or being terminal.

\begin{theorem}\label{thm:meta}
Let $P\in X$ be the germ of a smooth threefold and $\fm$ be the maximal ideal in $\sO_X$ defining $P$. Fix a positive rational number $q$. Let $\sP$ be a property of canonical pairs $(X,\fa^q)$ for ideals $\fa$ on $X$. Then the following statements are equivalent.
\begin{enumerate}
\item\label{itm:maximal}
Fix a non-negative rational number $s$. Then there exists a positive integer $l$ depending only on $q$ and $s$ such that if $\fa$ is an ideal on $X$ satisfying that $(X,\fa^q)$ is canonical and has the property $\sP$, then there exists a divisor $E$ over $X$ which computes $\mld_P(X,\fa^q\fm^s)$ and satisfies the inequality $a_E(X)\le l$.
\item\label{itm:mprimary}
Fix a non-negative rational number $s$ and a positive integer $b$. Then there exists a positive integer $l$ depending only on $q$, $s$ and $b$ such that if $\fa$ and $\fb$ are ideals on $X$ satisfying that $(X,\fa^q)$ is canonical and has the property $\sP$ and that $\fb$ contains $\fm^b$, then there exists a divisor $E$ over $X$ which computes $\mld_P(X,\fa^q\fb^s)$ and satisfies the inequality $a_E(X)\le l$.
\end{enumerate}
\end{theorem}

\begin{proof}
\textit{Step} 1.
The (\ref{itm:maximal}) follows from the special case of (\ref{itm:mprimary}) when $b=1$. It is necessary to derive (\ref{itm:mprimary}) from (\ref{itm:maximal}). Let $\cS=\{(\fa_i,\fb_i)\}_{i\in\bN}$ be an arbitrary sequence of pairs of ideals on $X$ such that $(X,\fa_i^q)$ is canonical and has the property $\sP$ and such that $\fb_i$ contains $\fm^b$. Assuming the (\ref{itm:maximal}), it is sufficient to find an integer $l$ such that for infinitely many $i$, there exists a divisor $E_i$ over $X$ which computes $\mld_P(X,\fa_i^q\fb_i^s)$ and satisfies the inequality $a_{E_i}(X)\le l$. Note Remark \ref{rmk:independent}. By Theorem \ref{thm:nonpos}, we may assume that $\mld_P(X,\fa_i^q\fb_i^s)$ is positive. Then by Corollary \ref{crl:mult}, there exists a positive integer $b_0$ depending only on $q$ and $s$ such that $\ord_{G_i}\fm\le b_0$ for every divisor $G_i$ over $X$ computing $\mld_P(X,\fa_i^q\fb_i^s)$.

\medskip
\textit{Step} 2.
We construct a generic limit $(\sfa,\sfb)$ of $\cS$ using the notation in Section \ref{sct:limit}. The $(\sfa,\sfb)$ is the generic limit with respect to a family $\cF=(Z_l,(\fa(l),\fb(l)),N_l,s_l,t_l)_{l\ge l_0}$ of approximations of $\cS$. The $\sfa$ and $\sfb$ are ideals on $\hat P\in\hat X$ where $\hat X$ is the spectrum of the completion of the local ring $\sO_{X,P}\otimes_kK$. We let $\hat\fm$ denote the maximal ideal in $\sO_{\hat X}$. Note that $\hat\fm^b\subset\sfb$ by $\fm^b\subset\fb_i$. By Lemma \ref{lem:limtonak} and Remark \ref{rmk:limit}(\ref{itm:limitineq}), the existence of $l$ is reduced to the inequality $\mld_{\hat P}(\hat X,\sfa^q\sfb^s)\le\mld_P(X,\fa_i^q\fb_i^s)$ for any $i\in N_{l_0}$ after replacing $\cF$ with a subfamily. By Theorem \ref{thm:grthan1}, we may assume that $(\hat X,\sfa^q\sfb^s)$ has the smallest lc centre $\hat C$ which is regular and of dimension one.

Since $\sfb$ is $\hat\fm$-primary, $\hat C$ is also the smallest lc centre of $(\hat X,\sfa^q)$. In particular, $\mld_{\hat P}(\hat X,\sfa^q)\le1$ by Theorem \ref{thm:grthan1}(\ref{itm:case4}), while $\mld_{\hat P}(\hat X,\sfa^q)\ge1$ by Remark \ref{rmk:limit}(\ref{itm:limitineq}). Thus $\mld_{\hat P}(\hat X,\sfa^q)=1$.

\medskip
\textit{Step} 3.
We apply Theorem \ref{thm:wbu} to $(\hat X,\sfa^q)$. There exist a divisor $\hat E$ over $\hat X$ computing $\mld_{\eta_{\hat C}}(\hat X,\sfa^q)$ and a regular system $x_1,x_2,x_3$ of parameters in $\sO_{\hat X}$ such that $\hat E$ is obtained by the weighted blow-up of $\hat X$ with $\wt(x_1,x_2)=(w_1,w_2)$ for some coprime positive integers $w_1,w_2$. We take $x_3$ generally from $\hat\fm$ so that $\ord_{x_3}\sfa$ is zero, where $\ord_{x_3}$ stands for the order along the divisor on $\hat X$ defined by $x_3$. Note that $\hat C$ is geometrically irreducible.

We fix a positive integer $j$ such that
\begin{align*}
j>bb_0.
\end{align*}
Let $\hat f\colon\hat Y\to\hat X$ be the weighted blow-up with $\wt(x_1,x_2,x_3)=(jw_1,jw_2,1)$ and $\hat F$ be its exceptional divisor. By Remark \ref{rmk:wbu}, there exists a regular system $y_1,y_2,y_3$ of parameters in $\sO_{X,P}\otimes_kK$ such that $\hat f$ is also the weighted blow-up with $\wt(y_1,y_2,y_3)=(jw_1,jw_2,1)$ (after regarding $y_1,y_2,y_3$ as elements in $\sO_{\hat X}$).

Discussed in the paragraph prior to Lemma \ref{lem:relative}, after replacing $\cF$ with a subfamily, $\hat f$ is descended to a projective morphism $f_l\colon Y_l\to X\times Z_l$ for any $l\ge l_0$. One can assume that $y_1,y_2,y_3$ come from $\sO_{X,P}\otimes_k\sO_{Z_l}$ and that for any $i\in N_l$, their fibres $y_{1i},y_{2i},y_{3i}$ at $s_l(i)\in Z_l$ form a regular system of parameters in $\sO_{X,P}$. $Y_l$ is klt and the exceptional locus of $f_l$ is a $\bQ$-Cartier prime divisor $F_l$. The fibre $f_i\colon Y_i\to X$ of $f_l$ at $s_l(i)$ is the weighted blow-up of $X$ with $\wt(y_{1i},y_{2i},y_{3i})=(jw_1,jw_2,1)$ whose exceptional divisor is $F_i=F_l\times_{Y_l}Y_i$.

Since $(jw_1,jw_2,1)=j(w_1,w_2,0)+(0,0,1)$, one has the inequality
\begin{align*}
\ord_{\hat F}\sfa^q\ge j\ord_{\hat E}\sfa^q+\ord_{x_3}\sfa^q=j(w_1+w_2)
\end{align*}
using $\ord_{\hat E}\sfa^q=a_{\hat E}(\hat X)-a_{\hat E}(\hat X,\sfa^q)=w_1+w_2$. Equivalently, $a_{\hat F}(\hat X,\sfa^q)=a_{\hat F}(\hat X)-\ord_{\hat F}\sfa^q\le 1$. Hence $a_{\hat F}(\hat X,\sfa^q)=1$ by $\mld_{\hat P}(\hat X,\sfa^q)=1$ in Step 2.

\medskip
\textit{Step} 4.
Let $\hat Q$ be the closed point in $\hat F$ which lies on the strict transform of $\hat C$. For $i\in N_l$, let $Q_i$ be the closed point in $F_i$ which lies on the strict transform of the curve on $X$ defined by $(y_{1i},y_{2i})\sO_X$. Applying Lemma \ref{lem:relative} to $(\hat X,\sfa^q\sfb^s)$, $\hat f$, and $\hat Q$, one has only to treat the case when $\mld_P(X,\fa_i^q\fb_i^s)$ is computed by a divisor $G_i$ such that $c_{Y_i}(G_i)=Q_i$.

For such $G_i$, the inequality $\ord_{G_i}\fm\le b_0$ holds by Step 1. Thus,
\begin{align*}
\ord_{G_i}\fb_i\le\ord_{G_i}\fm^b\le bb_0<j\le jw_2&=\ord_{F_i}(y_{1i},y_{2i})\sO_X\\
&\le\ord_{F_i}(y_{1i},y_{2i})\sO_X\cdot\ord_{G_i}F_i\\
&\le\ord_{G_i}(y_{1i},y_{2i})\sO_X,
\end{align*}
whence $\ord_{G_i}\fb_i=\ord_{G_i}(\fb_i+(y_{1i},y_{2i})\sO_X)$.

By $\hat\fm^b\subset\sfb$, there exists a non-negative integer $b'$ at most $b$ satisfying that
\begin{align*}
\sfb+(y_1,y_2)\sO_{\hat X}=\hat\fm^{b'}+(y_1,y_2)\sO_{\hat X}.
\end{align*}
Then one can assume that $\fb(l)+(y_1,y_2)\sO_{X\times Z_l}=(\fm^{b'}+\fm^l)\sO_{X\times Z_l}+(y_1,y_2)\sO_{X\times Z_l}$ for any $l\ge l_0$, which derives the inclusion $\fm^{b'}\subset\fb_i+\fm^l+(y_{1i},y_{2i})\sO_X$. One may assume that $l_0\ge b$, then $\fm^{b'}\subset\fb_i+(y_{1i},y_{2i})\sO_X$. Thus, $\ord_{G_i}\fb_i=\ord_{G_i}(\fb_i+(y_{1i},y_{2i})\sO_X)\le\ord_{G_i}\fm^{b'}$. In particular,
\begin{align*}
\mld_P(X,\fa_i^q\fm^{sb'})\le a_{G_i}(X,\fa_i^q\fm^{sb'})\le a_{G_i}(X,\fa_i^q\fb_i^s)=\mld_P(X,\fa_i^q\fb_i^s).
\end{align*}

\medskip
\textit{Step} 5.
We want the inequality $\mld_{\hat P}(\hat X,\sfa^q\sfb^s)\le\mld_P(X,\fa_i^q\fb_i^s)$, as seen in Step 2. Applying our assumption (\ref{itm:maximal}) in the case when the exponent of $\fm$ is one of $0,s,2s,\ldots,bs$, there exists a positive integer $l'$ depending only on $q$, $s$ and $b$ such that $\mld_P(X,\fa_i^q\fm^{sb'})$ is computed by a divisor $E_i$ satisfying the inequality $a_{E_i}(X)\le l'$. One has that $\ord_{E_i}\fa_i=q^{-1}(a_{E_i}(X)-a_{E_i}(X,\fa_i^q))\le q^{-1}l'$, so $E_i$ computes $\mld_P(X,(\fa_i+\fm^e)^q\fm^{sb'})$ for any integer $e$ at least $q^{-1}l'$, which equals $\mld_P(X,\fa_i^q\fm^{sb'})$. Together with Lemma \ref{lem:resolution}, one obtains that $\mld_{\hat P}(\hat X,\sfa^q\hat\fm^{sb'})=\mld_P(X,\fa_i^q\fm^{sb'})$ for any $i\in N_{l_0}$ after replacing $\cF$ with a subfamily. Hence by Step 4, the problem is reduced to showing the equality
\begin{align*}
\mld_{\hat P}(\hat X,\sfa^q\sfb^s)=\mld_{\hat P}(\hat X,\sfa^q\hat\fm^{sb'}).
\end{align*}

The equality $\sfb+(y_1,y_2)\sO_{\hat X}=\hat\fm^{b'}+(y_1,y_2)\sO_{\hat X}$ tells that
\begin{align*}
\sfb+(y_1,y_2,y_3^j)\sO_{\hat X}=\hat\fm^{b'}+(y_1,y_2,y_3^j)\sO_{\hat X}.
\end{align*}
Since $(y_1,y_2,y_3^j)\sO_{\hat X}=\hat f_*\sO_{\hat Y}(-j\hat F)=\sI_{\hat C}+\hat\fm^j$, where $\sI_{\hat C}$ is the ideal sheaf of $\hat C$, one concludes that $\sfb+\sI_{\hat C}=\hat\fm^{b'}+\sI_{\hat C}$, using $\hat\fm^j\subset\hat\fm^b\subset\sfb$ and $\hat\fm^j\subset\hat\fm^{b'}$. Therefore, the required equality follows from the precise inversion of adjunction, Corollary \ref{crl:pia}.
\end{proof}

Precise inversion of adjunction compares the minimal log discrepancy of a pair and that of its restricted pair by adjunction. Let $P\in X$ be the germ of a normal variety and $S+B$ be an effective $\bR$-divisor on $X$ such that $S$ is a normal prime divisor which does not appear in $B$. Suppose that they form a pair $(X,S+B)$, then one has the adjunction $K_X+S+B|_S=K_S+B_S$ in which $B_S$ is the different on $S$ of $B$.

\begin{conjecture}[Precise inversion of adjunction]\label{cnj:pia}
Notation as above. Then one has that $\mld_P(X,S+B)=\mld_P(S,B_S)$.
\end{conjecture}

This conjecture is regarded as the more precise version of Theorem \ref{thm:ia}. At present we know two cases when it holds. One is when $X$ is smooth \cite{EMY03} or more generally has lci singularities \cite{EM04}. The other is when the minimal log discrepancy is at most one \cite{BCHM10}, that is,

\begin{theorem}\label{thm:pia}
Conjecture \textup{\ref{cnj:pia}} holds when $(X,\Delta)$ is klt for some boundary $\Delta$ and $\mld_P(X,S+B)$ is at most one.
\end{theorem}

\begin{proof}
It is enough to show the inequality $\mld_P(X,S+B)\ge\mld_P(S,B_S)$. By inversion of adjunction, we may assume that $0<\mld_P(X,S+B)\le1$. Then by \cite[Corollary 1.4.3]{BCHM10}, there exists a projective birational morphism $\pi\colon Y\to X$ from a $\bQ$-factorial normal variety such that the divisorial part of its exceptional locus is a prime divisor $E$ computing $\mld_P(X,S+B)$. Let $S_Y$ and $B_Y$ denote the strict transforms of $S$ and $B$. Then the pull-back of $(X,S+B)$ is $(Y,S_Y+B_Y+bE)$ where $b=1-\mld_P(X,S+B)\ge0$. Let $C$ be an arbitrary irreducible component of $E\cap S_Y$. By $\mld_P(X,S+B)>0$, the $(Y,S_Y+B_Y+bE)$ is plt about the generic point $\eta_C$ of $C$. By adjunction, one can write
\begin{align*}
K_Y+S_Y+B_Y+bE|_{S_Y}=K_{S_Y}+B_{S_Y}
\end{align*}
about $\eta_C$. By \cite[Corollary 3.10]{Sh93}, $C$ has coefficient at least $b$ in $B_{S_Y}$. Hence, one obtains that $\mld_P(S,B_S)\le a_C(S_Y,B_{S_Y})\le1-b=\mld_P(X,S+B)$.
\end{proof}

\begin{lemma}\label{lem:pia}
Let $P\in X$ be the germ of a klt variety and $S$ be a prime divisor on $X$ such that $(X,S)$ is plt. Let $\Delta$ be the different on $S$ defined by $K_X+S|_S=K_S+\Delta$. Let $\hat X$ be the spectrum of the completion of the local ring $\sO_{X,P}$ and $\hat P$ be its closed point. Set $\hat S=S\times_X\hat X$ and $\hat\Delta=\Delta\times_X\hat X$. Let $\sfa$ be an $\bR$-ideal on $\hat X$ such that $\mld_{\hat P}(\hat X,\hat S,\sfa)\le1$. Then $\mld_{\hat P}(\hat X,\hat S,\sfa)=\mld_{\hat P}(\hat S,\hat\Delta,\sfa\sO_{\hat S})$.
\end{lemma}

\begin{proof}
Adding a high multiple of the maximal ideal $\hat\fm$ in $\sO_{\hat X}$ to each component of $\sfa$, we may assume that $\sfa$ is $\hat\fm$-primary. Then $\sfa$ is the pull-back of an $\bR$-ideal $\fa$ on $X$. By Remark \ref{rmk:regular}, the assertion is reduced to the precise inversion of adjunction $\mld_P(X,S,\fa)=\mld_P(S,\Delta,\fa\sO_S)$ for varieties, which follows from Theorem \ref{thm:pia}.
\end{proof}

\begin{proposition}\label{prp:piaE}
Let $X$ be the spectrum of the ring of formal power series in three variables over a field $K$ of characteristic zero and $P$ be its closed point. Let $x_1,x_2$ be a part of a regular system of parameters in $\sO_X$ and $w_1,w_2$ be coprime positive integers. Let $Y\to X$ be the weighted blow-up of with $\wt(x_1,x_2)=(w_1,w_2)$, $E$ be its exceptional divisor, and $f$ be the fibre of $E\to X$ at $P$. Let $\Delta$ be the different on $E$ defined by $K_Y+E|_E=K_E+\Delta$. Let $\fa$ be an $\bR$-ideal on $X$ whose weak transform $\fa_Y$ on $Y$ is defined. Suppose that $a_E(X,\fa)$ is zero. Then $\mld_P(X,\fa)=\mld_f(E,\Delta,\fa_Y\sO_E)$.
\end{proposition}

\begin{proof}
By the regular base change, we may assume that $K$ is algebraically closed. Since $(Y,E,\fa_Y)$ is crepant to $(X,\fa)$, it is enough to prove that $\mld_f(Y,E,\fa_Y)=\mld_f(E,\Delta,\fa_Y\sO_E)$.

Extend the $x_1,x_2$ to a regular system $x_1,x_2,x_3$ of parameters in $\sO_X$ and set $X'=\Spec K[x_1,x_2,x_3]$. Then $Y$ is the base change of the weighted blow-up of $X'$ with $\wt(x_1,x_2)=(w_1,w_2)$. Thus, the equality $\mld_{\eta_f}(Y,E,\fa_Y)=\mld_{\eta_f}(E,\Delta,\fa_Y\sO_E)$ follows from Lemma \ref{lem:pia} by cutting by the strict transform of the divisor on $X$ defined by $x_1^{w_2}+\lambda x_2^{w_1}$ for a general member $\lambda$ in $K$. Together with $\mld_f(Y,E,\fa_Y)\le\mld_{\eta_f}(Y,E)=1$, it is sufficient to verify that $\mld_Q(Y,E,\fa_Y)=\mld_Q(E,\Delta,\fa_Y\sO_E)$ for any closed point $Q$ in $f$ such that $\mld_Q(Y,E,\fa_Y)\le1$, which follows from Lemma \ref{lem:pia} again.
\end{proof}

\begin{corollary}\label{crl:pia}
Let $X$ be the spectrum of the ring of formal power series in three variables over a field $K$ of characteristic zero and $P$ be its closed point. Let $\fa$, $\fb$ and $\fc$ be $\bR$-ideals on $X$. Suppose that $\mld_P(X,\fa)$ equals one and that $(X,\fa)$ has an lc centre $C$ of dimension one on which $\fb\sO_C=\fc\sO_C$. Then $\mld_P(X,\fa\fb)=\mld_P(X,\fa\fc)$.
\end{corollary}

\begin{proof}
We may assume that $C$ is not contained in the cosupport of $\fb\fc$, because otherwise $\mld_P(X,\fa\fb)=\mld_P(X,\fa\fc)=-\infty$. By Theorem \ref{thm:wbu}, there exist a divisor $E$ over $X$ computing $\mld_{\eta_C}(X,\fa)=0$ and a part $x_1,x_2$ of a regular system of parameters in $\sO_X$ such that $E$ is obtained by the weighted blow-up $Y$ of $X$ with $\wt(x_1,x_2)=(w_1,w_2)$ for some $w_1,w_2$. We may assume that the weak transform $\fa_Y$ on $Y$ of $\fa$ is defined. Then, the assertion follows from Proposition \ref{prp:piaE} by $\fb\sO_E=\fc\sO_E$.
\end{proof}

\begin{proof}[Proof of Theorem \textup{\ref{thm:first}}]
It is sufficient to derive Conjectures \ref{cnj:acc}, \ref{cnj:alc}, \ref{cnj:madic} and \ref{cnj:nakamura} from Conjecture \ref{cnj:product}. All $\bR$-ideals on the germ $P\in X$ of a smooth threefold in Conjectures \ref{cnj:acc} to \ref{cnj:nakamura} may be assumed to be $\fm$-primary, where $\fm$ is the maximal ideal in $\sO_X$ defining $P$. There exists an \'etale morphism from $P\in X$ to the germ $o\in\bA^3$ at origin of the affine space, by which any $\fm$-primary $\bR$-ideal $\fa$ on $X$ is the pull-back of some $\bR$-ideal $\fb$ on $\bA^3$ by Lemma \ref{lem:regular}. Thus for Conjectures \ref{cnj:acc} to \ref{cnj:nakamura}, one has only to consider $\bR$-ideals on the fixed germ $P\in X$ of a smooth threefold.

By Lemma \ref{lem:rational}, these conjectures are reduced to the case $I=\{1/n\}$ of Conjecture \ref{cnj:nakamura}, that is, for a fixed positive integer $n$, it is enough to find an integer $l$ such that if $\fa$ is an ideal on $X$, then there exists a divisor $E$ over $X$ which computes $\mld_P(X,\fa^{1/n})$ and satisfies the inequality $a_E(X)\le l$. By Theorem \ref{thm:nonpos}, we have only to consider those $\fa$ for which $\mld_P(X,\fa^{1/n})$ is positive. Then by Corollary \ref{crl:mult}, there exists a positive integer $b$ depending only on $n$ such that $\ord_E\fm\le b$ for every divisor $E$ over $X$ computing $\mld_P(X,\fa^{1/n})$.

Set $q=1/n$ and apply Theorem \ref{thm:canonical}. It is enough to bound $a_E(X)$ for those $\fa$ in the case (\ref{itm:reduced}) of Theorem \ref{thm:canonical}. Suppose this case and use the notation in Theorem \ref{thm:canonical}. Let $E$ be an arbitrary divisor over $Y$ which computes $\mld_Q(Y,\Delta,\fa_Y^q)$, that equals $\mld_P(X,\fa^q)$. Then,
\begin{align*}
a_E(X)=a_E(Y)+\sum_F(a_F(Y)-1)\ord_EF\le a_E(Y)+(c-1)\sum_F\ord_EF,
\end{align*}
in which the summation takes over all exceptional prime divisors on $Y$, and
\begin{align*}
\sum_F\ord_EF\le\ord_E\fm\le b.
\end{align*}
Thus the boundedness of $a_E(X)$ is reduced to that of $a_E(Y)$. In other words, it is sufficient to treat the divisors computing $\mld_Q(Y,\Delta,\fa_Y^q)$. The ideal $\fb=\sO_Y(-n\Delta)$ in $\sO_Y$ is defined since $n\Delta$ is integral, for which $(Y,\fa_Y^q\fb^q)$ is crepant to $(Y,\Delta,\fa_Y^q)$. The $\fb$ satisfies that $\ord_E\fb\le n\sum_F\ord_EF\le nb$ by $\sO_Y(-n\sum F)\subset\fb$. In particular, $E$ computes $\mld_Q(Y,\fa_Y^q(\fb+\fn^{nb})^q)$ as well as $\mld_Q(Y,\Delta,\fa_Y^q)$ for the maximal ideal $\fn$ in $\sO_Y$ defining $Q$.

Replacing the notation $(Y,\fa_Y^q(\fb+\fn^{nb})^q)$ with $(X,\fa^q\fb^q)$ and $\fn$ with $\fm$, Conjectures \ref{cnj:acc} to \ref{cnj:nakamura} follow from the boundedness of $a_E(X)$ for some divisor $E$ over $X$ which computes $\mld_P(X,\fa^q\fb^q)$ such that $\mld_P(X,\fa^q)\ge1$ and such that $\fm^{nb}\subset\fb$. One may assume that $\fa$ is $\fm$-primary. Then one can apply Theorem \ref{thm:meta} with the property $\sP$ being empty, which reduces the boundedness of $a_E(X)$ to Conjecture \ref{cnj:product}.
\end{proof}

\section{Boundedness results}
In this section, we shall prove Conjecture \ref{cnj:product} in several cases. First we treat the case when either $(X,\fa^q)$ is terminal or $s$ is zero.

\begin{proof}[Proof of Theorem \textup{\ref{thm:terminal}}]
Let $\cS=\{\fa_i\}_{i\in\bN}$ be an arbitrary sequence of ideals on $X$ such that $(X,\fa_i^q)$ is terminal. We construct a generic limit $\sfa$ of $\cS$. We use the notation in Section \ref{sct:limit}, so $\sfa$ is the generic limit with respect to a family $\cF=(Z_l,\fa(l),N_l,s_l,t_l)_{l\ge l_0}$ of approximations of $\cS$, and $\sfa$ is an ideal on $\hat P\in\hat X$. We let $\hat\fm$ denote the maximal ideal in $\sO_{\hat X}$. To see the assertion (\ref{itm:terminal}), by Lemma \ref{lem:limtonak} and Remark \ref{rmk:independent}, it is enough to show the equality $\mld_{\hat P}(\hat X,\sfa^q\hat\fm^s)=\mld_P(X,\fa_i^q\fm^s)$ for any $i\in N_{l_0}$ after replacing $\cF$ with a subfamily. One has that $\mld_{\hat P}(\hat X,\sfa^q)>1$ by Remark \ref{rmk:limit}(\ref{itm:limitineq}). Thus $(\hat X,\sfa^q)$ satisfies the case \ref{cas:case1}, \ref{cas:case2} or \ref{cas:case3} in Theorem \ref{thm:grthan1}, which derives that $(\hat X,\sfa^q\hat\fm^s)$ does not have the smallest lc centre of dimension one. Hence the required equality holds by Theorem \ref{thm:grthan1}.

For (\ref{itm:zero}), starting instead with $\cS$ such that $(X,\fa_i^q)$ is canonical, we need to show the equality $\mld_{\hat P}(\hat X,\sfa^q)=\mld_P(X,\fa_i^q)$. This holds in the cases other than the case \ref{cas:case4} in Theorem \ref{thm:grthan1}, so we may assume the case \ref{cas:case4}, in which $\mld_{\hat P}(\hat X,\sfa^q)\le1$. By $\mld_P(X,\fa_i^q)\ge1$, the required equality follows from Remark \ref{rmk:limit}(\ref{itm:limitineq}).
\end{proof}

We shall study the case in Theorem \ref{thm:second}(\ref{itm:half}) when the lc threshold of the maximal ideal is at most one-half. We prepare a useful criterion for identifying a divisor over a variety.

\begin{lemma}\label{lem:parallel}
Let $P\in X$ be the germ of a smooth variety and $E$ be a divisor over $X$. Let $x_1,\ldots,x_c$ be a part of a regular system of parameters in $\sO_{X,P}$ and $w_1,\ldots,w_c$ be positive integers. Let $Y\to X$ be the weighted blow-up with $\wt(x_1,\ldots,x_c)=(w_1,\ldots,w_c)$, $F$ be its exceptional divisor, and $H_i$ be the strict transform of the divisor on $X$ defined by $x_i$. Suppose that
\begin{itemize}
\item
$c_X(E)$ coincides with $c_X(F)$, and
\item
the vector $(w_1,\ldots,w_c)$ is parallel to $(\ord_Ex_1,\ldots,\ord_Ex_c)$.
\end{itemize}
Then the centre on $Y$ of $E$ is not contained in the union $\bigcup_{i=1}^cH_i$.
\end{lemma}

\begin{proof}
The idea has appeared already in \cite[Lemma 6.1]{K03}. We may assume that $w_1,\ldots,w_c$ have no common divisors. One computes that
\begin{align*}
\ord_Ex_i=\ord_EH_i+\ord_Fx_i\cdot\ord_EF=\ord_EH_i+w_i\ord_EF.
\end{align*}
The $\ord_EH_i$ is positive iff the centre $c_Y(E)$ lies on $H_i$. Since the intersection $\bigcap_{i=1}^cH_i$ is empty, at least one of $\ord_EH_i$ is zero. Because $(\ord_Ex_1,\ldots,\ord_Ex_c)$ is parallel to $(w_1,\ldots,w_c)$, one concludes that $\ord_EH_i$ is zero for every $i$, which proves the assertion.
\end{proof}

The next lemma plays a central role in the proof of Theorem \ref{thm:second}(\ref{itm:half}).

\begin{lemma}\label{lem:half}
Let $C$ be the spectrum of the ring of formal power series in one variable over $k$ and $P$ be its closed point. Let $X\to C$ be a smooth projective morphism of relative dimension one and $f$ be its fibre at $P$. Let $\Delta$ be an effective $\bR$-divisor on $X$, $Q$ be a closed point in $f$, and $t$ be a positive real number. Suppose that
\begin{itemize}
\item
$(K_X+\Delta)\cdot f=0$,
\item
$\mld_f(X,\Delta)=1$, and
\item
$\mld_Q(X,\Delta+tf)=0$.
\end{itemize}
Then $t$ is at least one-half. Moreover if $t$ equals one-half, then $\mld_Q(X,\Delta+sf)=1-2s$ for any non-negative real number $s$ at most one-half.
\end{lemma}

\begin{proof}
\textit{Step} 1.
Let $E$ be a divisor over $X$ which computes $\mld_Q(X,\Delta+tf)=0$. We define the coprime positive integers $w_1$ and $w_2$ so that the vector $(w_1,w_2)$ is parallel to $(\ord_Ef,\ord_E\fn)$, where $\fn$ is the maximal ideal in $\sO_X$ defining $Q$. Take a regular system $x_1,x_2$ of parameters in $\sO_{X,Q}$ such that $x_1$ defines $f$ and $x_2$ is a general member in $\fn$. We claim that the divisor $F$ obtained by the weighted blow-up $Y\to X$ with $\wt(x_1,x_2)=(w_1,w_2)$ computes $\mld_Q(X,\Delta+tf)$.

This claim can be verified in the same way as in \cite{K17}. Assuming that $a_F(X,\Delta+tf)$ is positive, we shall derive a contradiction. For $i=1,2$, let $H_i$ be the strict transform of the divisor defined on $X$ by $x_i$. By Lemma \ref{lem:parallel}, the centre on $Y$ of $E$ would be a closed point $R$ in $F\setminus(H_1+H_2)$. The pull-back of $(X,\Delta+tf)$ is $(Y,bF+\Delta_Y+tH_1)$ in which $\Delta_Y$ is the strict transform of $\Delta$ and $b=1-a_F(X,\Delta+tf)<1$. Thus $(Y,F+\Delta_Y)$ is not lc about $R$, so $(F,\Delta_Y|_F)$ is not lc about $R$ by inversion of adjunction. Remark that this inversion of adjunction on $R\in Y$ holds by Lemma \ref{lem:pia} because $Y\to X$ is the base change of the weighted blow-up of $\Spec k[x_1,x_2]$ with $\wt(x_1,x_2)=(w_1,w_2)$. This means that $\ord_R(\Delta_Y|_F)$ is greater than one.

One computes that
\begin{align*}
1&=-(K_Y+bF+\Delta_Y+tH_1)\cdot F+1\\
&\le-(K_Y+bF+tH_1)\cdot F-\ord_R(\Delta_Y|_F)+1\\
&<((w_1+w_2-1)+b-tw_1)(-F^2)=\frac{1}{w_1}+\frac{1-t}{w_2}-\frac{1-b}{w_1w_2}.
\end{align*}
Together with $w_1\ge w_2$ and $b<1$, one would obtain that $w_2=1$ and $tw_1<b$. But then $a_F(X,\Delta)=a_F(X,\Delta+tf)+t\ord_Ff=(1-b)+tw_1<1$, which contradicts that $\mld_f(X,\Delta)=1$.

\medskip
\textit{Step} 2.
We have seen that $F$ computes $\mld_Q(X,\Delta+tf)=0$. Then $a_F(X,\Delta)=t\ord_Ff=tw_1$. Since $\ord_Q\Delta\le1$ by $\mld_f(X,\Delta)=1$, one has that $\ord_F\Delta\le w_1$. Thus,
\begin{align*}
w_2\le w_1+w_2-\ord_F\Delta=a_F(X,\Delta)=tw_1.
\end{align*}
By $(K_X+\Delta)\cdot f=0$, one has that $(\Delta\cdot f)=2$. Hence
\begin{align*}
w_2^{-1}\ord_F\Delta=(\ord_F\Delta)(F\cdot H_1)\le(\Delta_Y+(\ord_F\Delta)F)\cdot H_1=(\Delta\cdot f)=2,
\end{align*}
where the inequality follows from the fact that $f$ does not appear in $\Delta$ by $\mld_f(X,\Delta)=1$. Thus $\ord_F\Delta\le2w_2$ and
\begin{align*}
w_1-w_2\le w_1+w_2-\ord_F\Delta=a_F(X,\Delta)=tw_1,
\end{align*}
that is, $(1-t)w_1\le w_2$.

We have obtained that $(1-t)w_1\le w_2\le tw_1$. Therefore $t\ge1/2$, and moreover if $t=1/2$, then $w_1=2w_2$ so $(w_1,w_2)=(2,1)$.

\medskip
\textit{Step} 3.
Suppose that $t=1/2$. Let $s$ be a non-negative real number at most one-half. It is necessary to show that $\mld_Q(X,\Delta+sf)=1-2s$. One has that $\mld_Q(X,\Delta+(1/2)f)=0$ and it is computed by $F$. In particular, $a_F(X,\Delta)=2^{-1}\ord_Ff=1$. By $\mld_f(X,\Delta)=1$, one obtains that $\mld_Q(X,\Delta)=1$ and it is also computed by $F$. Then Lemma \ref{lem:mld}(\ref{itm:mldequal}) provides that
\begin{align*}
\mld_Q(X,\Delta+sf)=(1-2s)\mld_Q(X,\Delta)+2s\mld_Q(X,\Delta+(1/2)f)=1-2s.
\end{align*}
\end{proof}

\begin{proposition}\label{prp:half}
Let $X$ be the spectrum of the ring of formal power series in three variables over a field $K$ of characteristic zero and $P$ be its closed point. Let $\fa$ be an $\bR$-ideal such that $\mld_P(X,\fa)$ equals one and such that $(X,\fa)$ has an lc centre of dimension one. Then one of the following holds for the maximal ideal $\fm$ in $\sO_X$.
\begin{enumerate}
\item\label{itm:eqhalf}
The $\mld_P(X,\fa\fm^s)$ equals $1-2s$ for any non-negative real number $s$ at most one-half.
\item\label{itm:grthanhalf}
The $\mld_P(X,\fa\fm^{1/2})$ is positive.
\end{enumerate}
\end{proposition}

\begin{proof}
We may assume that $K$ is algebraically closed. By Theorem \ref{thm:wbu}, there exist a divisor $E$ over $X$ and a part $x_1,x_2$ of a regular system of parameters in $\sO_X$ such that $a_E(X,\fa)=0$ and such that $E$ is obtained by the weighted blow-up $Y\to X$ with $\wt(x_1,x_2)=(w_1,w_2)$ for some $w_1,w_2$. We may assume that the weak transform $\fa_Y$ of $\fa$ is defined. Let $f$ be the fibre of $E\to X$ at $P$ and $\Delta$ be the different on $E$ defined by $K_Y+E|_E=K_E+\Delta$. Take an $\bR$-divisor $A_X=e^{-1}\sum_{i=1}^eA_i$ on $X$ for large $e$ in which $A_i$ are defined by general members in $\fa$. Let $A$ be the strict transform on $Y$ of $A_X$ and set $A_E=A|_E$. By Proposition \ref{prp:piaE}, one obtains that
\begin{align*}
\mld_P(X,\fa\fm^s)=\mld_f(E,\Delta+A_E+sf)
\end{align*}
for any non-negative real number $s$. This equality for $s=0$ supplies that $\mld_f(E,\Delta+A_E)=1$. In particular, $f$ does not appear in $\Delta+A_E$.

Let $t$ be the positive real number such that $\mld_P(X,\fa\fm^t)=0$. If $t>1/2$, then the case (\ref{itm:grthanhalf}) holds. Suppose that $t\le1/2$. Then $\mld_f(E,\Delta+A_E+tf)=0$ but $\mld_{\eta_f}(E,\Delta+A_E+tf)=1-t>0$, so there exists a closed point $Q$ in $f$ such that $\mld_Q(E,\Delta+A_E+tf)=0$. With $(K_E+\Delta+A_E)\cdot f=0$, one can apply Lemma \ref{lem:half} to $(E,\Delta+A_E)$, which derives that $t=1/2$ and $\mld_f(E,\Delta+A_E+sf)=1-2s$. Hence the case (\ref{itm:eqhalf}) holds.
\end{proof}

\begin{proof}[Proof of Theorem \textup{\ref{thm:second}(\ref{itm:half})}]
Let $\cS=\{\fa_i\}_{i\in\bN}$ be an arbitrary sequence of ideals on $X$ such that $(X,\fa_i^q)$ is canonical and such that $\mld_P(X,\fa_i^q\fm^{1/2})$ is not positive. We construct a generic limit $\sfa$ of $\cS$. We use the notation in Section \ref{sct:limit}, so $\sfa$ is the generic limit with respect to a family $\cF=(Z_l,\fa(l),N_l,s_l,t_l)_{l\ge l_0}$ of approximations of $\cS$, and $\sfa$ is an ideal on $\hat P\in\hat X$. We let $\hat\fm$ denote the maximal ideal in $\sO_{\hat X}$. By Lemma \ref{lem:limtonak} and Remarks \ref{rmk:limit}(\ref{itm:limitineq}) and \ref{rmk:independent}, it is enough to show the inequality $\mld_{\hat P}(\hat X,\sfa^q\hat\fm^s)\le\mld_P(X,\fa_i^q\fm^s)$ for any $i\in N_{l_0}$ after replacing $\cF$ with a subfamily. By Theorem \ref{thm:grthan1}, we may assume that $(\hat X,\sfa^q)$ has the smallest lc centre of dimension one, in which $\mld_{\hat P}(\hat X,\sfa^q)\le1$. By Remark \ref{rmk:limit}(\ref{itm:limitineq}) again, one obtains that $\mld_{\hat P}(\hat X,\sfa^q)=1$ from the canonicity of $(X,\fa_i^q)$.

Let $t$ be the positive rational number such that $\mld_{\hat P}(\hat X,\sfa^q\hat\fm^t)=0$. By Theorem \ref{thm:nonpos}, we may assume that $s<t$. By Remark \ref{rmk:limit}(\ref{itm:limitresult}), one has that $\mld_P(X,\fa_i^q\fm^t)=0$ for any $i\in N_{l_0}$ after replacing $\cF$ with a subfamily. In particular, $t\le1/2$ by $\mld_P(X,\fa_i^q\fm^{1/2})\le0$. Applying Proposition \ref{prp:half} to $(\hat X,\sfa^q)$, one obtains that $t=1/2$ and $\mld_{\hat P}(\hat X,\sfa^q\hat\fm^s)=1-2s$. Thus $\mld_P(X,\fa_i^q\fm^{1/2})=0$. By Lemma \ref{lem:mld}(\ref{itm:mldconvex}), one obtains that
\begin{align*}
\mld_P(X,\fa_i^q\fm^s)&\ge(1-2s)\mld_P(X,\fa_i^q)+2s\mld_P(X,\fa_i^q\fm^{1/2})\\
&\ge1-2s=\mld_{\hat P}(\hat X,\sfa^q\hat\fm^s).
\end{align*}
\end{proof}

\begin{proof}[Proof of Corollary \textup{\ref{crl:main}}]
We may assume that $\fa$ is $\fm$-primary. By Theorems \ref{thm:terminal}(\ref{itm:terminal}) and \ref{thm:second}(\ref{itm:half}), we have only to consider ideals $\fa$ such that $\mld_P(X,\fa^q)$ equals one and such that $\mld_P(X,\fa^q\fm^{1/2})$ is positive. By Lemma \ref{lem:mld}(\ref{itm:mldconvex}), such $\fa$ satisfies that
\begin{align*}
\mld_P(X,\fa^q\fm^{1/n})\ge\Bigl(1-\frac{2}{n}\Bigr)\mld_P(X,\fa^q)+\frac{2}{n}\mld_P(X,\fa^q\fm^{1/2})>1-\frac{2}{n},
\end{align*}
whence $\mld_P(X,\fa^q\fm^{1/n})\ge1-1/n$ since $\mld_P(X,\fa^q\fm^{1/n})$ belongs to $n^{-1}\bZ$.

Let $E$ be an arbitrary divisor over $X$ which computes $\mld_P(X,\fa^q)$. Then
\begin{align*}
1-\frac{1}{n}\le\mld_P(X,\fa^q\fm^{1/n})\le a_E(X,\fa^q\fm^{1/n})=a_E(X,\fa^q)-\frac{1}{n}\ord_E\fm\le1-\frac{1}{n},
\end{align*}
so $\mld_P(X,\fa^q\fm^{1/n})=1-1/n$ and it is computed by $E$. By Lemma \ref{lem:mld}(\ref{itm:mldequal}), $E$ also computes $\mld_P(X,\fa^q\fm^s)$ for any non-negative real number $s$ at most $1/n$. Thus Corollary \ref{crl:main} follows from Theorem \ref{thm:terminal}(\ref{itm:zero}).
\end{proof}

By a similar argument, one can prove Conjecture \ref{cnj:product} in the opposite case when the lc threshold of the maximal ideal is at least one.

\begin{proof}[Proof of Theorem \textup{\ref{thm:second}(\ref{itm:one})}]
Let $\cS=\{\fa_i\}_{i\in\bN}$ be an arbitrary sequence of ideals on $X$ such that $(X,\fa_i^q)$ is canonical and such that $(X,\fa_i^q\fm)$ is lc. It is enough to show the existence of a positive integer $l$ such that for infinitely many indices $i$, there exists a divisor $E_i$ over $X$ which computes $\mld_P(X,\fa_i^q\fm^s)$ and satisfies the equality $a_{E_i}(X)=l$. Note Remark \ref{rmk:independent}. As in the proof of Theorem \ref{thm:second}(\ref{itm:half}), we construct a generic limit $\sfa$ on $\hat P\in\hat X$ of $\cS$ with respect to a family $\cF=(Z_l,\fa(l),N_l,s_l,t_l)_{l\ge l_0}$ of approximations of $\cS$. By Lemma \ref{lem:limtonak} and Theorem \ref{thm:grthan1}, we may assume that $(\hat X,\sfa^q)$ has the smallest lc centre of dimension one, in which $\mld_{\hat P}(\hat X,\sfa^q)\le1$. By Remark \ref{rmk:limit}(\ref{itm:limitineq}), one has that $\mld_{\hat P}(\hat X,\sfa^q)=1$.

Let $\hat E$ be a divisor over $\hat X$ which computes $\mld_{\hat P}(\hat X,\sfa^q)$. As in Remark \ref{rmk:descend}(\ref{itm:descendE}), replacing $\cF$ with a subfamily, one can descend $\hat E$ to a divisor $E_l$ over $X\times Z_l$ for any $l\ge l_0$. Writing $E_i$ for a component of the fibre of $E_l$ at $s_l(i)\in Z_l$, one may assume that $a_{E_i}(X)=a_{\hat E}(\hat X)$ and $a_{E_i}(X,\fa_i^q)=1$ for any $i\in N_l$. Then for any $i\in N_{l_0}$, $\mld_P(X,\fa_i^q)=1$ by the canonicity of $(X,\fa_i^q)$ and it is computed by $E_i$. By the log canonicity of $(X,\fa_i^q\fm)$, the $\ord_{E_i}\fm$ must equal one and $E_i$ also computes $\mld_P(X,\fa_i^q\fm)=0$. Therefore, $E_i$ computes $\mld_P(X,\fa_i^q\fm^s)$ by Lemma \ref{lem:mld}(\ref{itm:mldequal}).
\end{proof}

\section{Rough classification of crepant divisors}
By Theorems \ref{thm:terminal}(\ref{itm:terminal}) and \ref{thm:second}(\ref{itm:half}), for Conjecture \ref{cnj:product} one has only to consider ideals $\fa$ such that $\mld_P(X,\fa^q)$ equals one and such that $\mld_P(X,\fa^q\fm^{1/2})$ is positive. Then every divisor $E$ over $X$ computing $\mld_P(X,\fa^q)$ satisfies that $\ord_E\fm$ equals one. We close this paper by providing a rough classification of $E$.

\begin{theorem}\label{thm:crepant}
Let $P\in X$ be the germ of a smooth threefold and $\fm$ be the maximal ideal in $\sO_X$ defining $P$. Let $\fa$ be an $\bR$-ideal on $X$ such that $\mld_P(X,\fa)$ equals one. Let $E$ be a divisor over $X$ which computes $\mld_P(X,\fa)$ such that $\ord_E\fm$ equals one. Then there exist a regular system $x_1,x_2,x_3$ of parameters in $\sO_{X,P}$ and positive integers $w_1,w_2$ with $w_1\ge w_2$ such that for the weighted blow-up $Y$ of $X$ with $\wt(x_1,x_2,x_3)=(w_1,w_2,1)$, one of the following cases holds by identifying the exceptional divisor $F$ with $\bP(w_1,w_2,1)$ with weighted homogeneous coordinates $x_1,x_2,x_3$.
\begin{enumerate}[label=\textup{\arabic*.},ref=\arabic*]
\item\label{cas:toric}
$E$ equals $F$ as a divisor over $X$.
\item\label{cas:hypers}
The centre $c_Y(E)$ is the curve on $F$ defined by $x_1x_3^p+x_2^q$ for some positive integers $p$ and $q$ satisfying that $w_1+p=qw_2\le w_1+w_2$.
\item\label{cas:saturated}
The centre $c_Y(E)$ is the curve on $F$ defined by $x_1x_2+x_3^{w_1+w_2}$.
\end{enumerate}
\end{theorem}

\begin{proof}
\textit{Step} 1.
Let $w_1$ be the maximum of $\ord_Ex_1$ for all elements $x_1$ in $\fm\setminus\fm^2$. To see the existence of this maximum, let $Z\to X$ be the birational morphism from a smooth threefold $Z$ on which $E$ appears as a divisor. Applying Zariski's subspace theorem \cite[(10.6)]{Ab98} to $\sO_{X,P}\subset\sO_{Z,Q}$ for a closed point $Q$ in $E$, one has an integer $w$ such that $\sO_Z(-wE)_Q\cap\sO_{X,P}\subset\fm^2$. Then $\ord_Ex_1$ is less than $w$ for any $x_1\in\fm\setminus\fm^2$, so $w_1$ exists. Fix $x_1$ for which $\ord_Ex_1$ attains the maximum $w_1$.

Then let $w_2$ be the maximum of $\ord_Ex_2$ for those $x_2$ such that $x_1,x_2$ form a part of a regular system of parameters in $\sO_{X,P}$. Note that $w_1\ge w_2$. Fix $x_2$ for which $\ord_Ex_2$ attains the maximum $w_2$, and take a general member $x_3$ in $\fm$. Note that $\ord_Ex_3=\ord_E\fm=1$. The $x_1,x_2,x_3$ form a regular system of parameters in $\sO_{X,P}$. Let $Y$ be the weighted blow-up of $X$ with $\wt(x_1,x_2,x_3)=(w_1,w_2,1)$ and $F$ be its exceptional divisor. We identify $F$ with $\bP(w_1,w_2,1)$ with weighted homogeneous coordinates $x_1,x_2,x_3$.

\medskip
\textit{Step} 2.
Let $L$ be an arbitrary locus in $F$ defined by a weighted homogeneous polynomial of form either $u_1$, $u_2$ or $x_3$ where
\begin{itemize}
\item
$u_1=x_1+\sum_{i=0}^\rd{ w_1/w_2}\lambda_ix_2^ix_3^{w_1-iw_2}$ for some $\lambda_i\in k$,
\item
$u_2=x_2+\lambda x_3^{w_2}$ for some $\lambda\in k$.
\end{itemize}
We claim that the centre $c_Y(E)$ is not contained in any such $L$. Indeed by Lemma \ref{lem:parallel}, $c_Y(E)$ is not contained in the locus defined by $x_1x_2x_3$. In particular, $\ord_EF=\ord_E\fm=1$. If $L$ is defined by $u_i$ for $i=1$ or $2$, then the $u_i$, as an element in $\sO_X$, satisfies that
\begin{align*}
w_i\ge\ord_Eu_i\ge\ord_EL+\ord_Fu_i\cdot\ord_EF=\ord_EL+w_i.
\end{align*}
Thus $\ord_EL=0$, meaning that $c_Y(E)\not\subset L$. One also has that $\ord_Ex_i=\ord_Eu_i$. By Remark \ref{rmk:wbu}, we are free to replace $x_1$ with $u_1$ as well as $x_2$ with $u_2$ for the regular system $x_1,x_2,x_3$ of parameters constructed in Step 1.

\medskip
\textit{Step} 3.
Since an arbitrary closed point in $F$ lies on some $L$ in Step 2, the $c_Y(E)$ is either a curve or $F$ itself. The case $c_Y(E)=F$ is nothing but the case \ref{cas:toric}. We shall investigate the case when $c_Y(E)$ is an irreducible curve $C$ other than any $L$.

Let $d$ be the weighted degree of $C$ in $F\simeq\bP(w_1,w_2,1)$. We may assume that the weak transform $\fa_Y$ of $\fa$ is defined, so $(Y,bF,\fa_Y)$ is the pull-back of $(X,\fa)$ where $b=1-a_F(X,\fa)\le0$. Since
\begin{align*}
a_E(Y,F,\fa_Y)=a_E(Y,bF,\fa_Y)-(1-b)\ord_EF\le a_E(X,\fa)-(1-b)=b\le0,
\end{align*}
the $(Y,F,\fa_Y)$ is not plt about $\eta_C$. Thus $(F,\fa_Y\sO_F)$ is not klt about $\eta_C$ by inversion of adjunction, that is, $\ord_C(\fa_Y\sO_F)\ge1$. Thus the strict transform $A_Y$ of the $\bR$-divisor on $X$ defined by a general member in $\fa$ satisfies the inequality $A_Y|_F\ge C$. One computes that
\begin{align*}
d=w_1w_2C\cdot(-F)&\le dw_1w_2A_Y|_F\cdot(-F)=w_1w_2(\ord_F\fa)F^3\\
&=\ord_F\fa=w_1+w_2+1-a_F(X,\fa)\le w_1+w_2.
\end{align*}

\medskip
\textit{Step} 4.
Let $f$ be the weighted homogeneous polynomial in $x_1,x_2,x_3$ defining $C$, which has weighted degree $d\le w_1+w_2$. Since any weighted homogeneous polynomial in $x_2,x_3$ is decomposed as the product of polynomials of form $x_3$ or $x_2+\lambda x_3^{w_2}$ for some $\lambda\in k$, by Step 2 and the irreducibility of $C$, there exists a monomial which involves $x_1$ and appears in $f$.

Suppose that $d<w_1+w_2$. Then $x_1x_3^p$ is the only monomial of weighted degree $d$ involving $x_1$, in which $p=d-w_1\ge0$. The $p$ must be positive by Step 2, so $1\le p<w_2$. The $f$ is, up to constant, written as
\begin{align*}
f=x_1x_3^p+\sum_{i=0}^q\lambda_ix_2^ix_3^{w_1+p-iw_2}=\Bigl(x_1+\sum_{i=0}^{q-1}\lambda_ix_2^ix_3^{w_1-iw_2}\Bigr)x_3^p+\lambda_qx_2^qx_3^{w_1+p-qw_2}
\end{align*}
for some $\lambda_i\in k$, where $q=\rd{(w_1+p)/w_2}$. Since $f$ is irreducible, one has that $\lambda_q\neq0$ and $w_1+p=qw_2$. Replacing $f$ with $\lambda_q^{-1}f$ and $x_1$ with $\lambda_q^{-1}(x_1+\sum_{i=0}^{q-1}\lambda_ix_2^ix_3^{w_1-iw_2})$, $f$ is expressed as $x_1x_3^p+x_2^q$, which is the case \ref{cas:hypers}.

Suppose that $d=w_1+w_2$ and $w_1>w_2$. Then $x_1x_2$ and $x_1x_3^{w_2}$ are the only monomials of weighted degree $d$ involving $x_1$. If only $x_1x_3^{w_2}$ appears in $f$, then the case \ref{cas:hypers} holds by the same discussion as in the case $d<w_1+w_2$. If $x_1x_2$ appears in $f$, then the part in $f$ involving $x_1$ is, up to constant, written as $x_1(x_2+\lambda x_3)$ for some $\lambda\in k$. Replacing $x_2$ with $x_2+\lambda x_3$, one may write $f$ as
\begin{align*}
f=x_1x_2+\sum_{i=0}^q\lambda_ix_2^ix_3^{w_1+w_2-iw_2}=\Bigl(x_1+\sum_{i=1}^q\lambda_ix_2^{i-1}x_3^{w_1+w_2-iw_2}\Bigr)x_2+\lambda_0x_3^{w_1+w_2}
\end{align*}
for some $\lambda_i\in k$, where $q=\rd{w_1/w_2}+1$. One has that $\lambda_0\neq0$ by the irreducibility of $f$. Replacing $f$ with $\lambda_0^{-1}f$ and $x_1$ with $\lambda_0^{-1}(x_1+\sum_{i=1}^q\lambda_ix_2^{i-1}x_3^{w_1+w_2-iw_2})$, $f$ is expressed as $x_1x_2+x_3^{w_1+w_2}$, which is the case \ref{cas:saturated}.

Finally, suppose that $d=2w$ and $w_1=w_2=w$ for some $w$. If $w=1$, then $C$ must be a conic in $F\simeq\bP^2$, so the case \ref{cas:saturated} holds after replacing $x_1,x_2,x_3$. If $w\ge2$, then after replacing $x_1,x_2$ with their suitable linear combinations, we may assume that the part in $f$ not involving $x_3$ is either $x_1x_2$ or $x_2^2$. In the first case, $f$ is written as $f=(x_1+\lambda_1x_3^w)(x_2+\lambda_2x_3^w)+\lambda_3x_3^{2w}$ for some $\lambda_1,\lambda_2,\lambda_3\in k$. Then $\lambda_3\neq0$ by the irreducibility. Replacing $f$ with $\lambda_3^{-1}f$, $x_1$ with $\lambda_3^{-1}(x_1+\lambda_1x_3^w)$, and $x_2$ with $x_2+\lambda_2x_3^w$, $f$ is expressed as $x_1x_2+x_3^{2w}$, which is the case \ref{cas:saturated}. In the second case, $f$ is written as $f=(\lambda_1x_1+\lambda_2x_2+\lambda_3x_3^w)x_3^w+x_2^2$ for some $\lambda_1,\lambda_2,\lambda_3\in k$. Then $\lambda_1\neq0$ by the irreducibility. Replacing $x_1$ with $\lambda_1x_1+\lambda_2x_2+\lambda_3x_3^w$, $f$ is expressed as $x_1x_3^w+x_2^2$, which is the case \ref{cas:hypers}.
\end{proof}

One can compute the lc threshold of the maximal ideal in the case \ref{cas:saturated}.

\begin{proposition}
Suppose the case \textup{\ref{cas:saturated}} in Theorem \textup{\ref{thm:crepant}}. Then $(X,\fa\fm)$ is lc.
\end{proposition}

\begin{proof}
We keep the notation in Theorem \ref{thm:crepant}. $Y$ is the weighted blow-up of $X$ with $\wt(x_1,x_2,x_3)=(w_1.w_2,1)$ and $F$ is its exceptional divisor. For $i=1,2,3$, let $H_i$ be the strict transform of the divisor defined by $x_i$. Let $Q_i$ be the closed point in $F$ which lies on $H_j\cap H_k$ for a permutation $\{i,j,k\}$ of $\{1,2,3\}$. Let $C$ denote the centre on $Y$ of $E$, which is the curve defined in $F\simeq\bP(w_1,w_2,1)$ by $x_1x_2+x_3^{w_1w_2}$ of weighted degree $d=w_1+w_2$ in our case \ref{cas:saturated}. Let $A$ be the $\bR$-divisor defined by a general member in $\fa$ and $A_Y$ be its strict transform on $Y$.

We have seen in Step 3 of the proof of Theorem \ref{thm:crepant} that
\begin{align*}
d=w_1w_2C\cdot(-F)&\le dw_1w_2A_Y|_F\cdot(-F)=w_1+w_2+1-a_F(X,\fa)\le w_1+w_2.
\end{align*}
Since $d=w_1+w_2$, the above inequalities are eventually equalities, whence $a_F(X,\fa)=1$ and $A_Y|_F=C$.

The triple $(Y,F+A_Y+H_3)$ is crepant to $(X,A+H_{3X})$, where $H_{3X}$ is the divisor defined by $x_3$. Thus it is enough to show the log canonicity of $(Y,F+A_Y+H_3)$. Let $\Delta$ be the different on $F$ defined by $K_Y+F|_F=K_F+\Delta$. By inversion of adjunction, the log canonicity of $(Y,F+A_Y+H_3)$ is equivalent to that of $(F,\Delta+A_Y|_F+H_3|_F)$.

Let $g$ be the greatest common divisor of $w_1$ and $w_2$. Let $L\simeq\bP^1$ be the line in $F$ defined by $x_3$. Since $Y$ has a quotient singularity of type $\frac{1}{g}(1,-1)$ at $\eta_L$, the different $\Delta$ equals $(1-g^{-1})L$ as in Example \ref{exl:different}. Together with $A_Y|_F=C$ and $H_3|_F=g^{-1}L$, one has that $\Delta+A_Y|_F+H_3|_F=C+L$.

The assertion is reduced to the log canonicity of $(F,C+L)$, which can be checked directly by using the explicit expression $(x_1x_2+x_3^{w_1+w_2})x_3$ of the defining weighted polynomial of $C+L$. Along $L$, one can use inversion of adjunction again, which tells that $K_F+C+L|_L=K_L+Q_1+Q_2$.
\end{proof}

\section*{Acknowledgements}
I should like to thank Professor M. Musta\c{t}\u{a} for introducing me to the approach by using the generic limit of ideals to the ACC conjecture for minimal log discrepancies. I should also like to thank Dr.\ Y. Nakamura for the discussions on the relationship between (\ref{itm:madic}), (\ref{itm:nakamura}) and (\ref{itm:limit}) in Conjecture \ref{cnj:equiv}.

\end{document}